\documentclass[noinfoline]{imsart}
\RequirePackage[OT1]{fontenc}
\RequirePackage{amsthm,amsmath,natbib}
\setcitestyle{square,numbers}
\RequirePackage[colorlinks,citecolor=blue,urlcolor=blue]{hyperref}
\usepackage{enumerate}
\usepackage{amsmath}
\usepackage{amsfonts}

\usepackage{geometry}
\newtheorem{theorem}{Theorem}%[section]
\newtheorem{lemma}[theorem]{Lemma}
\newtheorem{proposition}[theorem]{Proposition}
\newtheorem{corollary}[theorem]{Corollary}
\theoremstyle{definition}
\newtheorem{definition}[theorem]{Definition}

\theoremstyle{remark}
\newtheorem{remark}[theorem]{Remark}
\numberwithin{equation}{section} 

\newcommand{\md}{\operatorname{d}}

\newcommand{\spann}[0]{\operatorname{span}}

\begin{document}

\begin{frontmatter}
\title{Reduction and reconstruction of SDEs via Girsanov and quasi Doob symmetries}
\runtitle{Reduction and reconstruction of SDEs}

\begin{aug}
\author{\fnms{Francesco C.} \snm{De Vecchi}\thanksref{a,e2}
\ead[label=e2,mark]{francesco.devecchi@uni-bonn.de}},
\author{\fnms{Paola} \snm{Morando}\thanksref{b,e3}%
\ead[label=e3,mark]{paola.morando@unimi.it}}
\and \author{\fnms{Stefania} \snm{Ugolini}\thanksref{c,e4}%
\ead[label=e4,mark]{stefania.ugolini@unimi.it}}

\affiliation{Rheinische Friedrich-Wilhelms-Universit\"at Bonn. and Universit\`a degli Studi di Milano}

\address[a]{Institute for Applied Mathematics and Hausdorff Center for Mathematics,
Endenicher Allee 60, 53115 Bonn, Germany.
\printead{e2}}

\address[b]{DISAA,
 Via Celoria 2, 20133 Milano, Italy.
\printead{e3}}

\address[c]{Dipartimento di Matematica,
 Via Saldini 50, 20113 Milano, Italy.
\printead{e4}}

\runauthor{De Vecchi, Morando and Ugolini}

\end{aug}

\begin{abstract}
A reduction procedure for stochastic differential equations based on stochastic symmetries including Girsanov random transformations is proposed. In this setting, a new notion of reconstruction is given, involving  the expectation values of functionals of solution to the SDE  and a reconstruction theorem for general stochastic symmetries is proved. Moreover, the notable case of reduction under the closed subclass of quasi Doob transformations is presented. The theoretical results are applied to stochastic models relevant in the applications.
\end{abstract}



\end{frontmatter}

\section{Introduction}

The study of symmetries and invariance properties of ordinary and partial differential equations (ODEs
and PDEs, respectively) is a classical and well-developed research field (see, e.g., \cite{Bluman,Olver,Stephani}) and provides a powerful tool for both computing some explicit solutions to the equations and analyzing their
qualitative behavior. Some important applications of this theory, in the case of ODEs, are the reduction of the dimension of a system of ODEs,
(see, e.g., \cite{Stephani}), or the development of symmetric numerical discretization schemes for the ODEs, which permit the preservation of  some important features of the dynamical system or the reduction of the numerical error of the approximation (see, e.g., \cite{Dorodnitsyn,HairerE2010}). \\

In recent years there has been a growing interest in generalizing and applying  techniques and results of the classical Lie's symmetry analysis to stochastic differential equations in both finite (SDEs) and infinite (SPDEs) dimensions. In this paper we use the approach of our previous works (see \cite{AlDeMoUg2018,AlDeMoUg2019,DMU2,DMU,DMU3}), where the concept of weak stochastic symmetry of a general SDE driven by semimartingale and some generalizations has been proposed (for different approaches to the same problem see, e.g., \cite{Applebaum2019markov,Zambrini2017,Li2010,Ga2019,Gaeta,GaSpa2017,Kozlov20182,Kozlov20181,Cruzeiro2016,Ortega20092,Ortega20091,Zambrini2008,Liaobook,Liao2019}).  \\
More precisely, our starting point is  the recent paper \cite{DMU3} on symmetries of SDEs driven by Brownian motion, where   we introduce the notion of general stochastic symmetry, i.e. an invariance of the set of solutions to an SDE with respect to transformations involving a space diffeomorphism, a stochastic time rescaling, a random rotation of the driving Brownian motion and a Girsanov transformation of the underlying probability measure. Thanks to this general notion of symmetry, in \cite{DMU3}  we  establish a one-to-one correspondence between the generalized weak symmetries of an SDE and the (deterministic) Lie's point symmetries of the related Kolmogorov PDE.\\
In this paper we face the problem of reduction and reconstruction by quadratures of an SDE admitting general symmetries, generalizing the results of \cite{DMU2}, where the same problem was discussed using only the weak symmetries introduced in \cite{DMU} which do not include the measure  change. \\
The proofs of reduction and reconstruction procedures proposed here are constructive and they can inspire concrete algorithms for the integration by quadratures of a symmetric SDE. A similar procedure, considering a smaller family of transformations, was addressed before by some authors (see, e.g., \cite{Gaeta,Ko2010,Ortega20092,Liaobook}) and in our previous papers \cite{AlDeMoUg2019,DMU2}.\\

The described procedure is interesting both from a theoretical and an applied perspective.
Indeed, from a theoretical point of view, the procedure provides a standard method to express the expectation of functionals of the solution process to symmetric Brownian-motion-driven SDEs, by using only Brownian motions, iterated integrals of Brownian motions and random time changes depending on the previous expressions. In this regard, this paper can be seen as the analogue of \cite{Craddock2007,Craddock2009,Craddock2004}, where  the symmetries of Kolmogorov equation are used for (more or less) explicitly computing the transition density or the expectation of special functionals of the process. In this paper a similar result is obtained by using the stochastic symmetries of the SDE which are, in general,  a wider class (see \cite{DMU3}). \\
From an applied point of view this is the first step in the direction of constructing symmetry adapted numerical methods for SDEs admitting symmetries involving probability measure change (see \cite{DeUg2017}, where this idea is applied to SDEs admitting only strong symmetries and \cite{AlDeMoUg2019}, where the topic is discussed in the case of weak stochastic symmetries without measure changes). \\
Moreover, in order to provide a suitable framework for numerical applications, we introduce  the notion of quasi Doob transformation,  which is a general stochastic transformation where the change of probability measure can be expressed as a Markovian function of the process plus a Riemann integral of a Markovian function of the process with respect to the time. We  prove that this family of stochastic transformations is closed with respect to composition and  that the reduction and reconstruction procedures, exploiting quasi Doob stochastic symmetries, can be done  involving only quasi Doob transformations. This  result is particularly relevant in numerical applications, since quasi Doob transformations, in evaluating the stochastic measure change, involve only the numerical computation of It\^o integrals which can be numerically simulated more easily and therefore with a lower numerical error.  \\

We stress that the results of this paper are not just a straightforward generalization of those in \cite{DMU2}, since the bigger family of transformations appearing in the reduction process necessitates the introduction  of  a generalized concept of reconstruction of an SDE, which can no longer be in a pathwise sense but only in mean. Although the reconstruction  result of this paper is weaker than the one proposed in \cite{DMU2}, it is strong  enough to be interesting in numerical applications (in particular using the subfamily of quasi Doob symmetries) and it permits to cover a wider class of symmetric SDEs that previously could be only partially tackled with the use of the associated Kolmogorov equation. For example, in Section \ref{section_CIR}, we apply our results to the study of the CIR model, which, despite admitting  a very symmetric Kolmogorov equation (see, e.g., \cite{Craddock2004}), has not weak stochastic symmetries of the form proposed in \cite{DMU2,DMU}. Furthermore the model discussed in Section \ref{section_2d}, which in \cite{DMU2} was only reduced by one dimension, in the new framework can be completely integrated.

The paper is organized as follows: in Section \ref{section_1} we recall the definition of general stochastic transformation for an SDE and its solutions, while in Section \ref{section_2}, after recalling   the definition of general stochastic symmetry for an SDE given in \cite{DMU3}, we prove that the family of general stochastic infinitesimal symmetries of an SDE forms a Lie algebra.
In Section \ref{section_3} we introduce the notion of quasi Doob transformation and  we define quasi Doob symmetries for an SDE, proving that they provide an interesting subclass of general stochastic symmetries. Section \ref{section_4} is devoted to prove reduction and reconstruction theorems in this new framework and in Section \ref{section 5}
our results are explicitly applied to some relevant  examples. \\
Einstein summation convention on repeated indices is used throughout the paper.

\section{General stochastic transformations for SDEs \label{section_1}}

Let $M, M'$ be  open subsets of $\mathbb{R}^n$.  Fixing a finite time horizon $[0,\mathcal{T}]$ with $\mathcal{T}\geq0$, we consider  a filtered probability space $\left(\Omega,\mathcal{F}, (\mathcal{F}_t)_{t \in [0,\mathcal{T}])} ,\mathbb{P}\right)$. Let $X$ be a continuous stochastic process taking values in $M$  and $W=\left(W^1,\dots,W^m\right)=\left(W^\alpha \right)$ be an  $m$-dimensional $\mathcal{F}_t$- Brownian motion and let $\mu : M \to \mathbb{R}^n$ and $\sigma : M \to \operatorname{Mat}\left(n,m\right)$ be two smooth functions.

\begin{definition}
The process $(X,W)$ solves (in a weak sense) the SDE with coefficients $\mu,\sigma$ (shortly solves the SDE $\left(\mu,\sigma\right)$) if, for all $t \in [0,\mathcal{T}]$,

\[X_{t}^i -X_0^i=\int_{0}^{t  } \mu ^i \left(X_s\right)ds + \int_{0}^{t }\sigma _\alpha ^i \left(X_s\right)dW_s^\alpha \quad i=1,\dots, n.\]
In the integral relation the processes $(|\mu\left(X_s\right)|^{\frac{1}{2}})_{s\in[0,\mathcal{T}]}$ and $\left(\sigma\left(X_s \right)\right)_{s\in[0,\mathcal{T}]}$ are supposed to belong to the class $M^2_{loc}([0,\mathcal{T}])$, i.e. to the class of processes  $\left(Y_s\right)_{s\in[0,\mathcal{T}]}$ that are progressively measurable  and such that $\int_{0}^{t}Y_s^2\left(\omega \right)ds < \infty \,  \operatorname{for} \, \operatorname{almost} \,  \operatorname{every} \,
 \omega \in \Omega$ and $t\in[0,\mathcal{T}]$.
\end{definition}

In the following, we recall the four different transformations for the solution processes to an SDE
that have been introduced in \cite{DMU2,DMU,DMU3}.

\subsection*{Spatial transformations }

Given an autonomous SDE $(\mu, \sigma)$ and the corresponding infinitesimal generator defined by
\begin{equation}\label{eqGENERATORL}
L= \frac{1}{2}\left(\sigma \sigma^T\right)^{ij} \partial_i \partial_j+ \mu^i \partial_i,
\end{equation}
we can consider a  diffeomorphism  $\Phi:M \to M'$ and its action on the component $X$ of the process.
Denoting by $\nabla \Phi : M \to \operatorname{Mat}\left(n,m\right)$ the Jacobian matrix
\[	 (\nabla \Phi)_j^i=\partial_j \Phi^i.\]
and  applying  It\^o formula (see, e.g., \cite{RoWi2000} Section 32 or \cite{Oksendal} Chapter 4) we have the following result.

\begin{proposition}\label{Prop1}
Given a diffeomorphism $\Phi : M \to M'$, if the process $\left(X,W\right)$ is solution to the  SDE $\left(\mu,\sigma\right)$, then the process $\left(\Phi\left(X\right),W\right)$ is solution to the SDE $\left(\mu',\sigma'\right)$ with
\begin{eqnarray*}
\mu'&=  L\left(\Phi\right)\circ \Phi^{-1}\\
\sigma'&=\left(\nabla \Phi \cdot \sigma\right)\circ\Phi^{-1}
\end{eqnarray*}
\end{proposition}
We remark that the above transformation is only a spatial transformation which does not change the driving Brownian motion.

\subsection*{Random time changes }

 Given a smooth and strictly positive density
 $\eta : M \to \mathbb{R}_+$, we denote by $H_\eta $ the transformation\[ t'=\int_0^t \eta_s(X_s)ds.\]
 $H_\eta $ is a Markovian absolutely continuous random time change acting on both components of the solution process $(X,W)$. The inverse random time change can be defined as
 \begin{equation}
  \label{alpha}
 \alpha_t=\inf\left\{ s\in \mathbb{R}_+\middle| \int_0^s \eta_\tau(X_\tau)d\tau > t \right\}.
 \end{equation}
If $W'_t$  is the solution to
\[%
	dW'_t=\sqrt{\eta\left(X_t\right)}dW_t,
\]%
then $H_\eta\left(W'\right)$ is again a Brownian motion and the  following proposition holds.

\begin{proposition}\label{Prop2}
Let $\eta : M \to \mathbb{R}_+$ be a smooth and strictly positive function and let $\left(X,W\right)$ be a solution to the SDE $\left(\mu,\sigma\right)$. Then the process $\left(H_\eta\left(X\right),H_\eta \left(W'\right)\right)$ is solution to the SDE $\left(\mu',\sigma'\right)$ with
\begin{eqnarray*}
	\mu'&=&\frac{1}{\eta}\mu\\
	\sigma'&=&\frac{1}{\sqrt{\eta}}\sigma.
\end{eqnarray*}
\end{proposition}

\subsection*{Random rotations}

Considering the well-known invariance under random rotations of Brownian motion (see \cite{DMU} and \cite{AlDeMoUg2018} for the general concept of gauge transformation), it is quite natural to consider also  the random rotation of the driving Brownian motion of the SDE. In fact, exploiting  the  notion of weak solution and
L\'evy characterization of Brownian motion, we obtain the following result.

\begin{proposition}
	Let $B : M \to SO \left(m\right)$ be a smooth function and let $\left(X,W\right)$ be a solution to the SDE $\left(\mu,\sigma\right)$. Then $\left(X,W'\right)$, where
	\[%
	dW'_t=B\left(X_t\right)\cdot dW_t,
	\]%
	is a solution to the SDE $\left(\mu',\sigma'\right)$ with
\begin{eqnarray*}	\mu'&=&\mu,\\
	\sigma'&=&\sigma \cdot B^{-1}.
	\end{eqnarray*}
	
\end{proposition}

\subsection*{Random changes of measure}
In order to further enlarge our class of transformations, we can exploit Girsanov theorem in order to introduce also a random change of the probability measure under which the driven process is a Brownian motion.\\

Given a process $\left(\theta_s\right)_{s\in[0,\mathcal{T}]} \in M^2_{loc}[0,\mathcal{T}] $, let us define the process  $\left(Z_t\right)_{t\in[0,\mathcal{T}]}$ by setting
\begin{equation}\label{supermartingale}
Z_t=Z_t(\theta):=\exp\left\{\int_{0}^{t}\theta_sdW_s -\frac{1}{2}\int_{0}^{t}\theta_s^2ds \right\}.
\end{equation}
An application of It\^o formula gives
\[ dZ_t=Z_t\theta_tdW_t,\]
which says that $Z$ is a local martingale.
We recall the fundamental result allowing the change of a probability measure into an equivalent one.
\begin{theorem}[Girsanov's theorem]
Let $\left(Z_t\right)_{t\in[0,\mathcal{T}]}$ be the exponential supermartingale defined  in \eqref{supermartingale}. If $\left(Z_t\right)_{t\in[0,\mathcal{T}]}$ is a $\mathbb{P}$-martingale, then the process  $\left(\hat W_t\right)_{t \in [0,\mathcal{T}]}$ given by
\[	\hat W_t=W_t-\int_{0}^{t} \theta_sds\]
is an $ (\mathcal{F})_t$-Brownian motion with respect to the probability measure  $\mathbb{Q}$, where
\[ \left.\frac{\md\mathbb{Q}}{\md\mathbb{P}}\right|_{\mathcal{F}_{\mathcal{T}}}=Z_\mathcal{T}.\]
\end{theorem}
\begin{proof}
The proof can be found  in \cite{RoWi2000}, Theorem 38.5.
\end{proof}

In order to guarantee that the supermartingale $Z_t$ is a $\mathbb{P}$-(global) martingale one can use the well known Novikov condition (see \cite{ReYo1999}, Chapter VIII, Proposition 1.14). However, in this paper, we choose a different approach (see \cite{DMU3}),  based on the non explosiveness property both of the original SDE and of the transformed one.

\begin{definition} Let $\mu\colon{M}\to{\mathbb{R}^n}$ and $\sigma\colon{M}\to{Mat(n,m)}$ be two smooth functions. The SDE $(\mu,\sigma)$ is called \emph{non explosive} if any solution $(X,W)$ to $(\mu,\sigma)$ is defined for all times $t\geq0$.\\
A smooth vector field $h$ is called \emph{non explosive} for the non explosive SDE $(\mu,\sigma)$ if the SDE $(\mu+\sigma\cdot{h},\sigma)$ is a non explosive SDE.\\
A positive smooth function $\eta$ is called a \emph{non explosive} time change for the non explosive SDE $(\mu,\sigma)$ if the SDE $\Bigr(\frac{\mu}{\eta},\frac{\sigma}{\sqrt{\eta}}\Bigl)$ is non explosive.
\end{definition}

\begin{lemma}
\label{lemm:lea}
Let $(\mu,\sigma)$ be a non explosive SDE admitting a weak solution $(X,W)$ and let $h\colon{M}\to\mathbb{R}^n$ be a smooth non explosive vector field. Then the exponential supermartingale $(Z_t)_{t\in[0,\mathcal{T}]}$ associated with $\theta_t=h(X_t)$ is a $\mathbb{P}$-(global) martingale.
\end{lemma}
The following result shows how this probability measure change works.
\begin{theorem}
\label{theo:theg}
Let $(X,W)$ be a solution to the non explosive SDE $(\mu,\sigma)$ on the probability space $(\Omega,\mathcal{F},\mathbb{P})$ and let $h$ be a smooth non explosive vector field for $(\mu,\sigma)$. Then $(X,W')$ is a solution to the SDE $(\mu',\sigma')=(\mu+\sigma\cdot{h},\sigma)$ on the probability space $(\Omega,\mathcal{F},\mathbb{Q})$, where

\begin{align*}
W'_t &=  -\int_0^t h(X_s)ds+W_t,\\
\left.\frac{d\mathbb{Q}}{d\mathbb{P}}\right|_{\mathcal{F}_{\mathcal{T}}} &=  \exp \left( \int_0^\mathcal{T} h_j(X_s)dW_s^j-\frac{1}{2}\int_0^\mathcal{T} \sum_{j=1}^m(h_j(X_s))^2ds \right).
\end{align*}

\end{theorem}

\begin{definition}[General stochastic transformation] Given two open subsets $M$ and $M'$ of $\mathbb{R}^n$, a diffeomorphism $\Phi\colon{M}\to{M'}$ and the smooth functions $B\colon{M}\to{SO(m)}$, $\eta\colon{M}\to\mathbb{R}_+$ and $h\colon{M}\to\mathbb{R}^m$, we call $T=(\Phi,B,\eta,h)$ a \emph{(weak finite) general stochastic transformation}. If $B=I, \eta=1$ and $h=0$ we call $T$ a \emph{strong (finite)  stochastic transformation}.
\end{definition}
It is important to  remember that the previous transformation cannot be applied to a generic SDE and that non explosiveness conditions must be taken into account. In the following definition we describe how the random transformation $T$ acts on the solution process.
\begin{definition}
\label{defi:dea}
Let $T=(\Phi,B,\eta,h)$ be a general stochastic transformation. Let $X$ be a continuous stochastic process taking values in $M$ and $W$ be an $m$-dimensional Brownian motion on the space $(\Omega,\mathcal{F},\mathbb{P})$ such that the pair $(X,W)$ is a solution to the non explosive SDE $(\mu,\sigma)$. Given two smooth non explosive functions $h$ and $\eta$ for the same SDE, we can define the process $P_T(X,W)=(P_T(X),P_T(W))$, where $P_T(X)$ takes values in $M'$ and $P_T(W)$ is a Brownian motion on the space $(\Omega,\mathcal{F},\mathbb{Q})$.
The process components are given by
\begin{align*}
X'=P_T(X)&=\Phi(H_\eta(X)), \notag \\
%d\tilde W_t&=\sqrt{\eta(X_t)}B(X_t)(dW_t-h(X_t)dt), \notag \\
W'=P_T(W)&=H_\eta(\tilde W), \notag
%\frac{d\mathbb{Q}}{d\mathbb{P}}&=exp\Bigr(\int_0^\mathcal{T}h_j(X_s)dW_s^j-\frac{1}{2}\int_0^\mathcal{T}\sum_{j=1}^m(h_j(X_s))^2ds\Bigl). \notag
\end{align*}
where $\tilde W_t$ satisfies
\begin{align*}
d\tilde W_t&=\sqrt{\eta(X_t)}B(X_t)(dW_t-h(X_t)dt), \notag
\end{align*}
and
\begin{align}
\frac{d\mathbb{Q}}{d\mathbb{P}}\Bigg|_{\mathcal{F}_{\mathcal{T}}}&=\exp\Bigr(\int_0^\mathcal{T}h_j(X_s)dW_s^j-\frac{1}{2}\int_0^\mathcal{T}\sum_{j=1}^m(h_j(X_s))^2ds\Bigl). \label{eq:definitionQ}
\end{align}
We call $P_T(X,W)$ the \emph{transformed process} of $(X,W)$ with respect to $T$ and we call the function $P_T$ the \emph{process transformation} associated with $T$.
\end{definition}
If we focus only on the SDE and the action $E_T$ of the stochastic transformation $T$, we can define the transformed SDE $E_T(\mu,\sigma)$ without making any request on the non-explosiveness of the solution process.
\begin{definition}
\label{defi:deb}
Let $T=(\Phi,B,\eta,h)$ be a general stochastic transformation. Given two smooth functions $\mu\colon{M}\to{\mathbb{R}^n}$ and $\sigma\colon{M}\to{Mat(n,m)}$, we define the SDE $E_T(\mu,\sigma)=(E_T(\mu),E_T(\sigma))$ on $M'$  as
\begin{align*}
E_T(\mu)&=\Bigr(\frac{1}{\eta}[L(\Phi)+\nabla\Phi\cdot\sigma\cdot{h}]\Bigl)\circ\Phi^{-1}, \notag \\
E_T(\sigma)&=\Bigr(\frac{1}{\sqrt{\eta}}\nabla\Phi\cdot\sigma\cdot{B^{-1}}\Bigl)\circ\Phi^{-1}. \notag
\end{align*}
We call $E_T(\mu,\sigma)$ the \emph{transformed SDE} of $(\mu,\sigma)$ with respect to $T$ and we call the map $E_T$ the \emph{SDE transformation} associated with $T$.
\end{definition}
\begin{remark}
We note that a specific order according to which the transformations are applied was chosen.
\end{remark}
\begin{theorem}
\label{theo:thea}
Given a stochastic transformation $T=(\Phi,B,\eta,h)$ and a solution $(X,W)$ to the non explosive SDE $(\mu,\sigma)$ such that $E_T(\mu,\sigma)$ is non explosive, then $P_T(X,W)$ is solution to the SDE $E_T(\mu,\sigma)$.
\end{theorem}
In order to better understand the nature of general stochastic transformations we can take advantage of Lie group theory.
Let $G=SO(m)\times\mathbb{R}_+\times\mathbb{R}^m$ be the group of rototranslations with a scaling factor, whose elements $g=(B,\eta,h)$ can be identified with the matrices
\[
\begin{pmatrix}
\sqrt{\eta}B^{-1} & h \\
0 & 1
\end{pmatrix}
\]
If we consider the trivial principal bundle $\pi\colon{M\times{G}}\to{M}$ with structure group $G$, we can define the action of $G$ on $M\times{G}$ given by
\begin{align}
R_{g_2}\colon{M\times{G}}&\to{M\times{G}} \notag \\
(x,g_1)&\mapsto(x,g_1\cdot{g_2}). \notag
\end{align}
which leaves $M$ invariant, where the standard product in $G$ is $g_1\cdot{g_2}=(B_1,\eta_1,h_1)\cdot(B_2,\eta_2,h_2)=(B_2B_1,\eta_1\eta_2,\sqrt{\eta_1}B_1^{-1}h_2+h_1)$.
Given another trivial principal bundle $\pi'\colon{M'\times{G}}\to{M'}$, we say that a diffeomorphism $F\colon{M\times{G}}\to{M'\times{G}}$ is an \emph{isomorphism} if $F$ preserves the structures of principal bundles of both $M\times{G}$ and $M'\times{G}$, i.e. there exists a diffeomorphism $\Phi\colon{M}\to{M'}$ such that
\begin{align}
\pi'\circ F&=\Phi\circ\pi, \notag \\
F\circ{R_g}&=R_g\circ{F}, \notag
\end{align}
for any $g\in{G}$.
Since an isomorphism of the previous form is completely determined by its value on $(x,e)$ (where $e$ is the unit element of $G$), there is a natural identification between a stochastic transformation $T=(\Phi,B,\eta,h)$ and the isomorphism $F_T$ such that $F_T(x,e)=(\Phi(x),g)$, where $g=(B,\eta,h)$. Given two stochastic transformations $T_1=(\Phi_1,B_1,\eta_1,h_1)$ and $T_2=(\Phi_2,B_2,\eta_2,h_2)$, we can consider their composition
\begin{equation}\label{eq:composition}
T_2\circ{T_1}=\Bigr(\Phi_2\circ\Phi_1,(B_2\circ\Phi_1)\cdot{B_1},(\eta_2\circ\Phi_1)\eta_1,\sqrt{\eta_1}B_1^{-1}\cdot(h_2\circ\Phi_1)+h_1\Bigl),
\end{equation}
Moreover, the inverse transformation of $T=(\Phi,B,\eta,h)$ is given by
\begin{equation}
\label{inversetransformation}
T^{-1}=\Bigr(\Phi^{-1},(B\circ\Phi^{-1})^{-1},(\eta\circ\Phi^{-1})^{-1},-\frac{1}{\sqrt{\eta}}B\cdot({h}\circ\Phi^{-1})\Bigl).
\end{equation}
The following theorem shows the probabilistic counterpart in terms of SDEs and process transformations of the previous geometric identification.
\begin{theorem}
\label{theo:tha}
Let $T_1$ and $T_2$ be two stochastic transformations, let $(\mu,\sigma)$ be a non explosive SDE such that $E_{T_1}(\mu,\sigma)$ and $E_{T_2}(E_{T_1}(\mu,\sigma))$ are non explosive and let $(X,W)$ be a solution to the SDE $(\mu,\sigma)$ on the probability space $(\Omega,\mathcal{F},\mathbb{P})$. Then on the probability space $(\Omega,\mathcal{F},\mathbb{Q})$, we have
\begin{align}
P_{T_2}(P_{T_1}(X,W))&=P_{T_2\circ{T_1}}(X,W), \notag \\
E_{T_2}(E_{T_1}(\mu,\sigma))&=E_{T_2\circ{T_1}}(\mu,\sigma). \notag
\end{align}
\end{theorem}
Since the set of stochastic transformations is a group with respect to the composition $\circ$, we can consider the one parameter group $T_a=(\Phi_a,B_a,\eta_a,h_a)$ and the corresponding \emph{infinitesimal (general) transformation} $V=(Y,C,\tau,H)$ obtained in the usual way
\begin{align}
Y(x)=&\partial_a(\Phi_a(x))\vert_{a=0} \notag \\
C(x)=&\partial_a(B_a(x))\vert_{a=0} \notag \\
\tau(x)=&\partial_a(\eta_a(x))\vert_{a=0} \notag \\
H(x)=&\partial_a(h_a(x))\vert_{a=0}, \notag
\end{align}
where $Y$ is a vector field on $M$, $C\colon{M}\to\mathfrak{so}(m)$, $\tau\colon{M}\to\mathbb{R}$ and $H\colon{M}\to\mathbb{R}^m$ are smooth functions. If $V$ is of the form $V=(Y,0,0,0)$ we call $V$ a \emph{strong infinitesimal stochastic  transformation}.
Conversely, given $V=(Y,C,\tau,H)$ we can reconstruct the one parameter transformation group $T_a$ choosing $\Phi_a$, $B_a$ and $\eta_a$ as the one parameter solutions to the following system
\begin{align}
\partial_a(\Phi_a(x))=&Y(\Phi_a(x)) \notag \\
\partial_a(B_a(x))=&C(\Phi_a(x) \cdot B_a(x) \notag \\
\partial_a(\eta_a(x))=&\tau(\Phi_a(x))\eta_a(x) \notag
\end{align}
with initial condition $\Phi_0=id_M, B_0=I, \eta_0=1$.
Moreover,  using Theorem \ref{theo:tha} and the properties of the flow we obtain that $h_a$ satisfies
\begin{align}
h_{b+a}(x)=&\frac{1}{\sqrt{\eta_a(x)}}B_a^{-1}(x)\cdot{h_b}(\Phi_a(x))+h_a(x) \notag \\
\partial_b(h_{b+a}(x))=&\frac{1}{\sqrt{\eta_a(x)}}B_a^{-1}(x)\cdot\partial_b(h_b(\Phi_a(x))) \notag \\
\partial_b(h_{b+a}(x))\vert_{b=0}=&\partial_a(h_a(x))=\frac{1}{\sqrt{\eta_a(x)}}B_a^{-1}(x)\cdot{H(\Phi_a(x))}, \notag
\end{align}
with initial condition $h_0(x)=0$.
Finally, given a finite stochastic transformation $T=(\Phi,B,\eta,h)$ and an infinitesimal one, $V=(Y,C,\tau,H)$, with associated one parameter group given by $T_a$, we can write the \emph{pushforward} of $V$ through the transformation $T$ as
\begin{multline}
T_*(V)=\Bigr(\Phi_*(Y),(B\cdot{C}\cdot{B^{-1}}+Y(B)\cdot{B^{-1}})\circ\Phi^{-1},\Bigr(\tau+\frac{Y(\eta)}{\eta}\Bigl)\circ\Phi^{-1}, \\
\Bigr(-\frac{1}{\sqrt{\eta}}B\cdot\Bigr(-\frac{\tau}{2}+C\Bigl)\cdot{h}+\frac{1}{\sqrt{\eta}}B\cdot{H}+\frac{1}{\sqrt{\eta}}B\cdot{Y(h)}\Bigl)\circ\Phi^{-1}\Bigl),
\end{multline}
where the pushforward $\Phi_*(Y)$ of vector fields is defined as 
\begin{equation}\label{eq:pushforward}
\Phi_*(Y)=(\nabla \Phi \cdot Y)\circ \Phi^{-1}.
\end{equation}
Moreover, given two infinitesimal stochastic transformations $V_1=(Y_1,C_1,\tau_1,H_1)$ and $V_2=(Y_2,C_2,\tau_2,H_2)$, we can consider their Lie brackets given by
\begin{multline}
[V_1,V_2]=\Bigr([Y_1,Y_2],Y_1(C_2)-Y_2(C_1)-\{C_1,C_2\},Y_1(\tau_2)-Y_2(\tau_1), \\
Y_1(H_2)-Y_2(H_1)-\Bigr(-\frac{\tau_1}{2}+C_1\Bigl)\cdot{H_2}+\Bigr(-\frac{\tau_2}{2}+C_2\Bigl)\cdot{H_1}\Bigl),
\end{multline}
where $\{C_1,C_2\}$ denotes the standard Lie brackets of matrices.

\section{General stochastic symmetries}\label{section_2}

In this section we exploit the general stochastic transformations introduced in Section \ref{section_1} in order to enlarge the class of symmetries for SDEs.

\begin{definition}[Finite and infinitesimal (general) symmetry] A (general) stochastic transformation $T$ is a \emph{(finite weak general) symmetry} of a non explosive SDE $(\mu,\sigma)$ if, for every solution process $(X,W)$, $P_T(X,W)$ is a solution process to the same SDE.
An infinitesimal (general) stochastic transformation $V$ generating a one parameter group $T_a$ is called an \emph{infinitesimal (general) symmetry} of the non explosive SDE $(\mu,\sigma)$ if $T_a$ is a symmetry of $(\mu,\sigma)$.
\end{definition}
\begin{proposition}
\label{prop:pra}
A stochastic transformation $T=(\Phi,B,\eta,h)$ is a symmetry of the non explosive SDE $(\mu,\sigma)$ if and only if
\begin{align}
\Bigr(\frac{1}{\eta}[L(\Phi)+\nabla\Phi\cdot\sigma\cdot{h}]\Bigl)\circ\Phi^{-1}&=\mu \notag \\
\Bigr(\frac{1}{\sqrt{\eta}}\nabla\Phi\cdot\sigma\cdot{B^{-1}}\Bigl)\circ\Phi^{-1}&=\sigma \notag
\end{align}
\end{proposition}
Next theorem provides the \emph{general determining equations} satisfied by the infinitesimal symmetries of an SDE $(\mu, \sigma)$.
\begin{theorem}\label{TheoDetEq}
An infinitesimal stochastic transformation $V=(Y,C,\tau,H)$ is an infinitesimal symmetry of the non explosive SDE $(\mu,\sigma)$ if and only if $V$ generates a one parameter group defined on $M$ and the following equations hold
\begin{align}
\label{equa:eqa}
Y(\mu)-L(Y)-\sigma\cdot{H}+\tau\mu&=0 \\
\label{equa:eqb}
[Y,\sigma]+\frac{1}{2}\tau\sigma+\sigma\cdot{C}&=0.
\end{align}
\end{theorem}
In order to prove that the set of (general) infinitesimal symmetries of a non explosive SDE $(\mu, \sigma)$ is a Lie algebra, we  need the following technical Lemma.
\begin{lemma}
\label{lemma:lea}
Given an infinitesimal stochastic symmetry $V=(Y,C,\tau,H)$ of the non explosive SDE $(\mu,\sigma)$, for any smooth function $f\in{C^\infty(M)}$, we have
\begin{equation}\label{eq_lemmared1}
Y(L(f))-L(Y(f))=-\tau{L(f)}+\nabla(f)\cdot\sigma\cdot{H}\\
\end{equation}
\begin{equation}\label{eq_lemmared2}
Y(\sigma^T)\cdot \nabla (f)=\sigma^T \cdot (\nabla Y)^T \cdot \nabla(f)-\frac 12 \tau \sigma^T\cdot  \nabla(f)+C\cdot \sigma^T \cdot \nabla (f),
%\nabla(Y(f))\cdot\sigma-Y(\nabla(f)\cdot\sigma)&=\frac{1}{2}\tau\nabla(f)\cdot\sigma+\nabla(f)\cdot\sigma\cdot{C}, \notag \\
\end{equation}
where $L=A^{ij}\partial_{ij}+\mu^i\partial_i$ (with  $A=\frac{1}{2}\sigma\cdot\sigma^T$) is the infinitesimal generator defined in \eqref{eqGENERATORL}, $Y=Y^i\partial_i$ and  $(\nabla Y)^i_k=\partial_kY^i$.
\end{lemma}
\begin{proof}
 In order to prove \eqref{eq_lemmared1}, let us consider
%$Y=Y^i\partial_i$ and $L=A^{ij}\partial_{ij}+\mu^i\partial_i$, where $A=\frac{1}{2}\sigma\cdot\sigma^T$.
%We have
\begin{align}
Y(L(f))-L(Y(f))=&Y^i\partial_i(A^{jk}\partial_{jk}(f)+\mu^j\partial_j(f)) \notag \\
&-(A^{jk}\partial_{jk}(Y^i\partial_i(f))+\mu^j\partial_j(Y^i\partial_i(f)))= \notag \\
=&(Y^i\partial_i(A^{jk})-A^{ik}\partial_i(Y^j)-A^{ji}\partial_i(Y^k))\partial_{jk}(f) \notag \\
&+(Y^i\partial_i(\mu^j)-A^{ik}\partial_{ik}(Y^j)-\mu^i\partial_i(Y^j))\partial_j(f). \notag
\end{align}
Thus, we can rewrite the claim as
\begin{align}
\label{equat:eqx}
Y^i\partial_i(\mu^j)-A^{ik}\partial_{ik}(Y^j)-\mu^i\partial_i(Y^j)&=-\tau\mu^j+\sigma^{j}_kH^k, \\
\label{equat:eqw}
Y^i\partial_i(A^{jk})-A^{ik}\partial_i(Y^j)-A^{ji}\partial_i(Y^k)&=-\tau(A^{jk}).
\end{align}
Equation \eqref{equat:eqx} holds owing to the first determining equation \eqref{equa:eqa}. Moreover, if we consider
 the (right) product of the second determining equation \eqref{equa:eqb} with $\sigma^T$
\begin{equation*}
[Y, \sigma]\cdot\sigma^T+\frac 12 \tau\sigma \cdot \sigma^T+\sigma\cdot C \cdot \sigma^T=0
\end{equation*}
and we add this equation with its transpose, since $C$ is an antisymmetric matrix, we get \eqref{equat:eqw}.\\
In order to prove \eqref{eq_lemmared2}, we start by considering the transpose of the second determining equation \eqref{equa:eqb}. If we multiply on the right by $\nabla(f)$ we get
\begin{equation}\label{eq_lemmared3}
[Y, \sigma]^T\cdot \nabla (f)+\frac 12 \tau \sigma^T \cdot \nabla (f)-C \cdot \sigma^T\cdot \nabla(f)=0.
\end{equation}
Since, by definition, we have
\begin{equation*}
[Y, \sigma]^T=[Y^i\partial_i(\sigma^k_{\alpha})]^T-[\partial_i Y^k \sigma^i_{\alpha}]^T=Y(\sigma^T)-\sigma^T(\nabla Y)^T,
\end{equation*}
we can rewrite equation \eqref{eq_lemmared3} as
\begin{equation*}
Y(\sigma^T)\cdot \nabla (f)-\sigma^T \cdot (\nabla Y)^T \cdot \nabla (f) +\frac 12 \tau \sigma^T \cdot \nabla (f) -C \cdot \sigma^T \cdot \nabla (f)=0
\end{equation*}
 and we get the thesis.
\end{proof}
\begin{theorem}\label{TheoCommunatorSymm}
Given two (general) infinitesimal symmetries  $V_1$ and $V_2$ of a non explosive SDE $(\mu, \sigma)$, the commutator $[V_1,V_2]$ is an infinitesimal symmetry of $(\mu, \sigma)$.
\end{theorem}
\begin{proof}
We report here a sketch of the proof. It can also be found in \cite{Cattarin}. Given two symmetries $V_i=(Y_i, C_i, \tau_i, H_i)$, $i=1,2$ we have to prove that their commutator $[V_1,V_2]$  satisfies the two determining equations given in Theorem \ref{TheoDetEq}. Let us consider equation \eqref{equa:eqa} for the commutator
\begin{align*}
&[Y_1,Y_2](\mu) - L([Y_1,Y_2])-\sigma\cdot\left( Y_1(H_2)-Y_2(H_1) \right) \notag\\
-&\sigma\cdot\left( \frac 12 \tau_1 H_2 -C_1\cdot H_2
-\frac 12 \tau_2 H_1+C_2\cdot H_1 \right)+[Y_1(\tau_2)-Y_2(\tau_1)]\mu=0\notag\\
\end{align*}
that can be rewritten as
\begin{align*}
&Y_1\left( Y_2(\mu)-\sigma\cdot H_2+\tau_2 \mu\right) - Y_2\left( Y_1(\mu)-\sigma\cdot H_1+\tau_1 \mu \right) \notag \\
&- L([Y_1,Y_2])+ \tau_1 Y_2(\mu)-\tau_2 Y_1(\mu) \notag \\
&+ \left( Y_1(\sigma)-\frac 12 \tau_1 \sigma + \sigma \cdot C_1 \right)\cdot H_2 - \left( Y_2(\sigma)-\frac 12 \tau_2\sigma + \sigma \cdot C_2 \right)\cdot H_1=0 \notag\\
\end{align*}
Exploiting the determining equations for $V_1$ and $V_2$ we get
\begin{align*}
&Y_1\left(L(Y_2)\right) - Y_2\left( L(Y_1)  \right)- L([Y_1,Y_2]) \notag \\
&+ \tau_1 Y_2(\mu)-\tau_2 Y_1(\mu) + \left(\nabla Y_1\cdot \sigma- \tau_1 \sigma \right)\cdot H_2 - \left(\nabla Y_2\cdot \sigma- \tau_2 \sigma \right)\cdot H_1=0\notag\\
\end{align*}
and using Lemma \ref{lemma:lea} we find
\begin{align*}
&L(Y_1\left(Y_2)\right) -\tau_1L(Y_2) +\nabla Y_2\cdot \sigma\cdot H_1-L(Y_2\left(Y_1)\right) +\tau_2L(Y_1) -\nabla Y_1\cdot \sigma\cdot H_2+ \notag \\
& -L([Y_1,Y_2])+\nabla Y_1\cdot \sigma \cdot H_2 -\tau_1\sigma \cdot H_2 -\nabla Y_2\cdot \sigma \cdot H_1 -\tau_2\sigma \cdot H_1- \tau_2 Y_1(\mu)+ \tau_1 Y_2(\mu)=0. \notag \\
\end{align*}
This equation holds due to the first determining equation for $V_1$ and $V_2$.
The proof of the second determining equation for the commutator $[V_1,V_2]$ does not involve the components $H_i$ and can be found in \cite{DMU}.
\end{proof}

\begin{remark}
The previous Theorem  shows that the family of general symmetries of an SDE  is a Lie algebra (i.e. the commutator of two symmetries can be expressed as a linear combinations of other symmetries with constant coefficients). This requirement is exactly one of the hypotheses of Theorem \ref{theor:tha} below, and will be  essential in the reduction and reconstruction of SDEs.
\end{remark}

\section{Quasi Doob transformations}\label{section_3}

Within the family of general transformations we identify the relevant class of quasi Doob transformations.  For the analogous important definition of generalized Doob transformations in an abstract setting see \cite{doob} and references therein.

\begin{definition}\label{quasidoob}[Quasi Doob transformation] Let $(\mu,\sigma)$ be a non explosive SDE and let $(X,W)$ be a solution to $(\mu,\sigma)$. Given a smooth function $h\colon M\to\mathbb{R}^m$ non explosive with respect to $(\mu,\sigma)$, we say that the stochastic transformation $(id_M,I,1,h)$ is a \emph{quasi Doob transformation with respect to the SDE $(\mu,\sigma)$} if there exists a smooth function  $\mathfrak{h}\colon M\to\mathbb{R}^m$ such that the transformed measure $\mathbb{Q}$ (which, by Definition \ref{defi:dea} of general stochastic transformation, is given by expression \eqref{eq:definitionQ}) satisfies the following condition
\[
\frac{d\mathbb{Q}}{d\mathbb{P}}\Bigg|_{\mathcal{F}_\mathcal{T}}=\exp\left\{\mathfrak{h}(X_\mathcal{T})-\mathfrak{h}(X_0)-\int_{0}^{\mathcal{T}}G_{\mathfrak{h}}((X_s))ds\right\}.
\]
where $\mathcal{T}$ is an arbitrary fixed time and $G_{\mathfrak{h}}$ is a suitable $C^2(\mathbb{R}^m)$ function depending on $\mathfrak{h}$.
\end{definition}

\begin{remark}
Hereafter we call the stochastic transformation $T=(id_M,I,1,h)$, satisfying Definition \ref{quasidoob}, a quasi Doob transformation characterized by the smooth function $\mathfrak{h}$. Furthermore, with an abuse of terminology, we call quasi Doob transformation with respect the SDE $(\mu,\sigma)$ any general stochastic transformation of the form $T' \circ T$, where $T=(id_M,I,1,h)$ is a quasi Doob transformation in the sense of Definition \ref{quasidoob}, and $T'$ is a stochastic transformation of the form $T'=(\Phi,B,\eta,0)$.
\end{remark}

The next result provides a necessary and sufficient explicit condition for the characterization of Girsanov transformations $h$ that are quasi Doob transformations associated with the function $\mathfrak{h}$. In the following we consider only conservative transformations. For non conservative (generalized) Doob transformations see \cite{doob}.
\begin{proposition}\label{PropDoobTransf}
Let $h\colon M\to\mathbb{R}^m$ be a smooth function associated with a random change of measure transformation on $(\mu,\sigma)$. Then $h$ is a quasi Doob transformation associated with the function $\mathfrak{h}\colon M\to\mathbb{R}$ if and only if the following conditions hold
\begin{align}
\label{equa:equata}
h_j(x)&=\sigma_j^i(x)\partial_{i}(\mathfrak{h})(x),\\
\label{equa:equatb}
\frac{1}{2}\sum_{j=1}^m(h_j(x))^2&=\frac{L(\exp(\mathfrak{h}))}{\exp(\mathfrak{h})}-L(\mathfrak{h})(x)=G_{\mathfrak{h}}(x)-L(\mathfrak{h})(x).
\end{align}
\end{proposition}
\begin{proof}
Matching the Radon-Nikodym derivative given in Theorem \ref{theo:theg} with the one in Definition \ref{quasidoob} we have
\[
\int_0^\mathcal{T}h_j(X_s)dW^j_s-\frac{1}{2}\int_0^\mathcal{T}\sum_{j=1}^m(h_j(X_s))^2ds=\mathfrak{h}(X_\mathcal{T})-\mathfrak{h}(X_0)-\int_{0}^{\mathcal{T}}G_{\mathfrak{h}}(X_s)ds.
\]
By applying It\^o formula to $\mathfrak{h}$ we get
\[
\mathfrak{h}(X_\mathcal{T})-\mathfrak{h}(X_0)=\int_0^\mathcal{T}L\mathfrak{h}(X_t)dt+\int_0^\mathcal{T}\nabla\mathfrak{h}(X_t)\sigma(X_t)dW_t.
\]
By uniqueness of the canonical semimartingale decomposition of continuous processes (see Section 31 and Definition 31.3 in \cite{RoWi2000}) and by uniqueness of the martingale representation theorem for processes adapted to Brownian filtrations (see Theorem 36.1 in \cite{RoWi2000}) we deduce the equality between the integrands of corresponding stochastic integrals, obtaining
\begin{align}
\label{equa:equata2}
h_j(x)&=\sigma_j^i(x)\partial_{i}(\mathfrak{h})(x),\\
\label{equa:equatb2}
\frac{1}{2}\sum_{j=1}^m(h_j(x))^2&=G_{\mathfrak{h}}(x)-L(\mathfrak{h}(x)).
\end{align}
In particular by \eqref{equa:equatb2} and \eqref{eqGENERATORL} we get
\begin{align}
G_{\mathfrak{h}}(x)&=\frac{1}{2}\sum_{j=1}^m(h_j(x))^2+\frac{1}{2}\left(\sigma \sigma^T\right)^{ij} \partial_i \partial_j\mathfrak{h}+ \mu^i \partial_i\mathfrak{h}\\
&=\frac{1}{2}\sum_{j=1}^m(\sigma_j^i(x)\partial_{i}(\mathfrak{h})(x))^2+\frac{1}{2}\left(\sigma \sigma^T\right)^{ij} \partial_i \partial_j\mathfrak{h}+ \mu^i \partial_i\mathfrak{h},
\end{align}
where we used \eqref{equa:equata2}. Since
\begin{equation}
\partial_i \partial_j\exp(\mathfrak{h})=\exp(\mathfrak{h})[\partial_i \partial_j\mathfrak{h}+(\partial_i\mathfrak{h})(\partial_j\mathfrak{h})]
\end{equation}
we have 
\begin{align*}
\frac{L(\exp(\mathfrak{h}))}{\exp(\mathfrak{h})}&=\frac{1}{2}\sum_{j=1}^m(\sigma_j^i(x)\partial_{i}(\mathfrak{h})(x))^2+\frac{1}{2}\left(\sigma \sigma^T\right)^{ij} \partial_i \partial_j\mathfrak{h}+ \mu^i \partial_i\mathfrak{h}\\
&=G_{\mathfrak{h}}(x).
\end{align*}
This prove that, if $h$ is a quasi Doob transformation, conditions \eqref{equa:equata} and \eqref{equa:equatb} hold. The converse is a consequence of It\^o formula applied to the function $\mathfrak{h}$.
\end{proof}
We remark that, according to Proposition \ref{PropDoobTransf}, the fact that a change of measure is a quasi  Doob transformation strongly depends on the SDE $(\mu,\sigma)$. In particular equation \eqref{equa:equata} depends on $\sigma$ and equation \eqref{equa:equatb} depends, through the operator $L$, on both $\mu$ and $\sigma$.\\

The next proposition states that quasi Doob transformations is, in some sense, a closed class with respect to composition of general stochastic transformations.

\begin{proposition}Let $(\mu,\sigma)$ be a non explosive SDE. Let $T_1=(\Phi_1,B_1,\eta_1,h_1)$ be a quasi Doob transformation with respect to $(\mu,\sigma)$ and let $T_2=(\Phi_2,B_2,\eta_2,h_2)$ be a quasi Doob transformation with respect to $E_{T_1}(\mu,\sigma)$. Then $T_2 \circ	T_1$ is a quasi Doob transformation with respect to the SDE $(\mu,\sigma)$.
\end{proposition}
\begin{proof}
	Since $T_1$ and $T_2$ are quasi Doob transformations, there exist two smooth functions $\mathfrak{h}_1, \mathfrak{h}_2$ such that
	\begin{equation}\label{Almostdoobconditions}
	h_1=\sigma^T\cdot \nabla (\mathfrak{h}_1), \quad h_2=\tilde \sigma^T\cdot \nabla (\mathfrak{h}_2)
	\end{equation}
	where $\tilde \sigma=\left(\frac{1}{\sqrt{\eta_1}}\nabla \Phi_1 \cdot \sigma \cdot B_1^T\right) \circ \Phi^{-1}_1$ is given by Definition \ref{defi:deb} (see also Theorem \ref{theo:thea}). In order to show that also the composition is a quasi Doob transformation we have to show that  the forth component  $\hat{h}$ of the stochastic composition $T_2 \circ T_1$
is the given by
	\begin{equation}\label{eq:hath}\hat{h}=\sigma^T \cdot \nabla(\mathfrak{h}_3)
	\end{equation}
	for some smooth function $\mathfrak{h_3}$. By equation \eqref{eq:composition} the expression of $\hat{h}$ is
	$$ \hat h= \sqrt{\eta_1} B_1^T \cdot (h_2  \circ{ \Phi}_1) + h_1.$$
	Since $T_1$ and $T_2$ are quasi Doob transformations, by \eqref{Almostdoobconditions}  we have
	$$\hat h=\sqrt{\eta_1} B_1^T \cdot ( (\tilde \sigma^T\cdot\nabla (\mathfrak{h}_2)) \circ \Phi_1)+ \sigma^T\cdot \nabla (\mathfrak{h}_1).$$
	Substituting  the expression of $\tilde \sigma^T\circ \Phi_1=\frac{1}{\sqrt{\eta_1}}B_1 \cdot \sigma^T \cdot (\nabla \Phi_1)^T$, using the fact that $B_1^TB_1=1$ and the chain rule for derivatives of composite functions we finally obtain
	$$\hat h= \sigma^T[ (\nabla \Phi_1)^T \cdot (\nabla (\mathfrak{h}_2)\circ \Phi_1) +   \nabla (\mathfrak{h}_1)]=\sigma^T\nabla(\mathfrak{h}_2 \circ \Phi_1+ \mathfrak{h}_1),$$
	and so $\hat{h}$ has the expression \eqref{eq:hath} for $\mathfrak{h}_3=\mathfrak{h}_2 \circ \Phi_1+ \mathfrak{h}_1$.
\end{proof}
If in equation \eqref{equa:equata}  both $h_a$ and $\mathfrak{h}_a$ depend on a parameter $a$ and we take the derivative with respect to that parameter in $a=0$ with initial condition $h_0=0$ and $\mathfrak{h}_0=0$, we obtain that there exists a function $k=\partial_a(\mathfrak{h}_a)\vert_{a=0}$ such that
\begin{align}
\label{equa:equatc}
H_j(x)&=\sigma_j^i(x)\partial_{x^i}(k)(x),
%\label{equa:equatd}
%0&=L(k).
\end{align}
Since quasi Doob transformations form a closed class in the sense of the previous proposition, the following Theorem
provides the determining equations for the infinitesimal symmetries of quasi Doob type.
\begin{theorem}\label{TheoDOOB}
An infinitesimal stochastic transformation $V=(Y,C,\tau,H)$ ( with $H=\sigma^T\cdot\nabla k$) is a symmetry of the SDE $(\mu,\sigma)$ involving only quasi Doob transformations with respect to $(\mu,\sigma)$ if and only if $V$ generates a one parameter group of transformations such that the following equations hold
\begin{align}
%H-\sigma^T\cdot\nabla k&=0 \notag \\
Y(\mu)-L(Y)-\sigma\cdot\sigma^T\cdot\nabla k+\tau\mu&=0 \notag \\
[Y,\sigma]+\frac{1}{2}\tau\sigma+\sigma\cdot C&=0 \notag \\
%L(k)&=0. \notag
\end{align}
\end{theorem}
In the following an infinitesimal stochastic symmetry satisfying the hypotheses of Theorem \ref{TheoDOOB}  is called a \emph{quasi Doob symmetries} for the SDE $(\mu, \sigma)$.
\begin{remark}
	When in Definition \ref{quasidoob} we take
	\[
	%L(k)=0
	G_{\mathfrak{h}}(x)=0
	\]
	we obtain the well-known class of Doob transformations. They provide a quite natural setting when we look for symmetries of an SDE. Indeed, in our previous paper \cite{DMU3} it was established  a one-to-one correspondence between infinitesimal symmetries of Doob type of an SDE and Lie's point infinitesimal symmetries of the corresponding Kolmogorov equation. In general the family of symmetries of an SDE can be  wider than the family  of the symmetries for the corresponding Kolmogorov equation as pointed out in \cite{DMU3}. The non Doob type symmetries denoted here as quasi Doob symmetries are a particularly interesting class. The other class of non Doob symmetries can be characterized as the class  of infinitesimal transformations $V=(Y,C,\tau, H)$ such that does not exist any function $k$ for which  $H= \sigma^T\cdot \nabla k $. The last family is not empty only for SDEs driven by  $m$ dimensional  Brownian motions with $m>1$, while for SDEs driven by one-dimensional Brownian motions  only quasi Doob symmetries can exist. For a complete discussion see \cite{DMU3}.
\end{remark}

In order to restrict our reduction and reconstruction procedure to quasi Doob symmetries we have to prove the following result.
\begin{theorem}\label{TheoCommunatorQuasiDoob}
Given two quasi Doob infinitesimal symmetries  $V_1=(Y_1,C_1,\tau_1, H_1)$ and $V_2=(Y_2,C_2,\tau_2, H_2)$ (where $H_i=\sigma^T \cdot \nabla k_i$) of a non explosive SDE $(\mu, \sigma)$, the commutator $[V_1,V_2]$ is an infinitesimal quasi Doob symmetry of $(\mu, \sigma)$. Moreover, if $V_1$ and $V_2$ are Doob symmetries, also $[V_1,V_2]$ is a Doob symmetry.
\end{theorem}
\begin{proof} Since, by Theorem \ref{TheoCommunatorSymm}, the commutator of two (general) stochastic infinitesimal symmetries is a general stochastic infinitesimal symmetry, and being the quasi Doob symmetries a special case of general symmetries, we have that the commutator of two quasi Doob symmetries of the SDE $(\mu,\sigma)$ is a general symmetry of the SDE $(\mu,\sigma)$. Therefore, we have only to prove that the commutator is also an infinitesimal quasi Doob transformation, namely that
\begin{align*}
H&=Y_1(H_2)-Y_2(H_1)-\left( -\frac 12 \tau_1+C_1\right)\cdot H_2+ \left( -\frac 12 \tau_2+C_2\right)\cdot H_1 \notag \\
&=\sigma^T\cdot \nabla k \notag
\end{align*}
for a suitable function $k$.
Using the fact $H_i= \sigma^T\cdot \nabla k_i$, we can write $H$ in the form
\begin{align*}
H&=Y_1(\sigma^T\cdot \nabla k_2)-Y_2(\sigma^T\cdot \nabla k_1)-\left( -\frac 12 \tau_1+C_1\right)\cdot\sigma^T\cdot \nabla k_2+ \left( -\frac 12 \tau_2+C_2\right)\cdot\sigma^T\cdot \nabla k_1 \notag
\end{align*}
and, by Lemma \ref{lemma:lea}, we find
\begin{align*}
H&=\sigma^T\cdot (\nabla Y_1)^T\cdot \nabla k_2-\frac 12 \tau_1\sigma^T \cdot \nabla k_2+C_1 \cdot \sigma^T \cdot \nabla k_2 +\sigma^T \cdot Y_1(\nabla k_2) \notag \\
-&\sigma^T\cdot (\nabla Y_2)^T\cdot \nabla k_1+\frac 12 \tau_2\sigma^T \cdot \nabla k_1-C_2 \cdot \sigma^T \cdot \nabla k_1 -\sigma^T \cdot Y_2(\nabla k_1) \notag \\
+& \frac 12 \tau_1\sigma^T \cdot \nabla k_2-C_1 \cdot \sigma^T \cdot \nabla k_2-\frac 12 \tau_2\sigma^T \cdot \nabla k_1+C_2 \cdot \sigma^T \cdot \nabla k_1 \notag \\
=&\sigma^T \cdot\nabla (Y_1(k_2)-Y_2(k_1))\notag
\end{align*}
Finally we have to prove that if $V_1$ and $V_2$ are Doob symmetries, then also $[V_1,V_2]$ is a Doob symmetry, i.e. if $L(k_i)=0$ for $i=1,2$, then $L(Y_1(k_2)-Y_2(k_1))=0$.
Using Lemma \ref{lemma:lea}  we get
\begin{align}
L(Y_1(k_2)-Y_2(k_1))=-\nabla k_2 \cdot \sigma\cdot \sigma^T \cdot \nabla k_1+ \nabla k_1 \cdot \sigma\cdot \sigma^T \cdot \nabla k_2
\end{align}
and this expression vanishes since $\sigma\cdot \sigma^T$ is a symmetric matrix. This concludes the proof.
\end{proof}

\section{Reduction and reconstruction }\label{section_4}

In this section we generalize the results of \cite{DMU2} providing a reduction and reconstruction scheme for SDEs admitting the general symmetries described above.\\
We start by introducing some useful tools, that are standard in the Lie symmetry analysis of deterministic differential equations. Secondly, we recall the notions of reduced SDE, reduced process and triangular SDE proposed in \cite{DMU2}. The idea is that, given an SDE defined on a $n$-dimensional smooth manifold $M$ and admitting a suitable Lie algebra of general infinitesimal symmetries, we can find a reduced SDE which is defined in a lower dimensional manifold. Furthermore, if the initial SDE is triangular, we can obtain the solution to the original SDE from the solution to the reduced one by using composition with smooth functions and Riemann and It\^o integrals (see also the more precise Definition \ref{standard_reconstruction}). Finally, we extend this idea to the general setting introduced in this paper, introducing a new notion of reconstruction for SDEs admitting infinitesimal symmetries with random change of measure (see Definition \ref{definition_integrable}).

\subsection{Some preliminary geometric results}

Let us start by recalling some definitions we need in the following.

\begin{definition}
Let $M$ be an open subset of $\mathbb{R}^n$  and $TM$ its tangent space. Given $k$ vector fields $Y_1,...,Y_k$  defined on $M$,  we say that $\Delta \subset TM$  is a distribution of constant rank $r\leq n$ generated by $Y_1,...,Y_k$ if
$$\Delta=\spann_{C^{\infty}(M)}\{Y_1,...,Y_k\},$$
and,  for any $x\in M$, the subspace
$$\Delta_x=\spann_{\mathbb{R}}\{Y_1(x),...,Y_k(x)\} \subset \mathbb{R}^n $$
has dimension $r$.
\end{definition}

\begin{definition}\label{definition_regular_vectors}
A set of vector fields $Y_1,...,Y_r$  is regular on $M$ if the distribution $\Delta=\spann_{C^{\infty}(M)}\{Y_1,...,Y_r\}$  has constant rank r, i.e. for any $x \in M$, the vectors $Y_1(x),...,Y_r(x)\in \mathbb{R}^n$ are linearly independent.
\end{definition}

\begin{definition}\label{definition_solvable_coordinate}
Let $Y_1,...,Y_r$ be a set of regular vector fields on $M$ which are generators of a solvable Lie algebra $\mathfrak{g}$. We say that $Y_1,...,Y_r$ are in \emph{canonical  form} if there are $i_1,...,i_l$ such that $i_1+...+i_l=r$ and, for any $x \in M$
$$(Y_1|...|Y_r)=\left(\begin{array}{c|c|c|c}
I_{i_1} & G^1_1(x) & ... & G^1_l(x) \\
\hline
0 & I_{i_2} & ... & G^2_l(x)\\
\hline
\vdots & \ddots & \ddots & \vdots \\
0 & 0 & ... & I_{i_l}\\
\hline
 0 & 0 & 0 & 0 \end{array} \right), $$
where  $G^h_k:M \rightarrow Mat(i_h,i_k)$ are smooth functions.
\end{definition}

\begin{theorem}\label{theorem_solvable_coordinate}
Let $\mathfrak{g}$ be a solvable Lie algebra on $M$ such that $\mathfrak{g}$ has constant rank $r$ as a  distribution of $TM$.  Then, for any  $x_0 \in M$, there is a set of generators $Y_1,...,Y_r$  of $\mathfrak{g}$ and a local diffeomorphism $\Phi:U(x_0) \rightarrow \tilde{M}$ such that $\Phi_*(Y_1),...,\Phi_*(Y_r)$ are generators in canonical form for $\Phi_*(\mathfrak{g})$ (where the pushforward $\Phi_*(Y_i)$ of vector field is defined in equation \eqref{eq:pushforward}).
\end{theorem}
\begin{proof}

The proof can be found in \cite{DMU2}, Theorem 2.6.
${}\hfill$
\end{proof}

\subsection{Reduction of SDEs via general symmetries}

In this section we introduce and discuss the definitions of reduced and triangular SDEs and we provide the main results for reduction of SDEs throughout general symmetries.

%\begin{definition}\label{definition_reduction}
Let $(\mu,\sigma)$ be an SDE defined on $M\subset \mathbb{R}^n$. The SDE $(\mu,\sigma)$  can be \emph{reduced with respect to the coordinates $x^{r+1},...,x^n$} (where $r\in\{1,...,n-1\}$) if
\[ \mu^j(x)=\mu^j(x^{r+1},...,x^n), \quad \sigma^j_{\alpha}(x)=\sigma^{j}_{\alpha}(x^{r+1},...,x^n),\]
for any $j=r+1,...,n$ and $\alpha=1,...,m$.
%\end{definition}

\begin{definition}\label{definition_reduction}
%\begin{remark}
If the SDE $(\mu,\sigma)$ defined on $M\subset \mathbb{R}^n$ can be reduced with respect to the coordinates $x^{r+1},...,x^n$, and we denote by $M'$ the open set obtained by projecting $M$ on the subspace generated by the coordinates $x^{r+1},...,x^n$, we can define a new SDE $(\mu',\sigma')$ on $M'$ such that
\[ \mu'^j(x)=\mu^{j+r}(x^{r+1},...,x^n), \quad \sigma'^j_{\alpha}(x)=\sigma^{j+r}_{\alpha}(x^{r+1},...,x^n)\]
where $j=1,...,n-r$. The SDE $(\mu',\sigma')$ on $M'$ is called \emph{the reduced SDE of $(\mu,\sigma)$ with respect to the coordinates $x^{r+1},...,x^n$}.

\end{definition}

\begin{remark}\label{remark_trinagular5}
If $(X,W)$ is a solution to an SDE which is reducible with respect to the variables $(x^{r+1},...,x^n)$, the process $(X',W)=((X^{r+1}_t,...,X^n_t),W)$ satisfies an SDE of the form
\[dX'^i_t=\mu^{i+r}(X^{1}_t,...,X^{n-r}_t)dt+\sum_{\alpha=1}^m\sigma^{i+r}_{\alpha}(X^{1}_t,...,X^{n-r}_t)dW^{\alpha}_t,
 \]
 for $i=1,...,n-r$. The process $(X',W)$ is called the reduced process.
\end{remark}

\begin{remark}\label{remarkreduction1}
If $Y_1,...,Y_r$ are a set of vector fields in canonical form such that $V_i=(Y_i, 0,0,0)$  are also strong symmetries of the  SDE $(\mu, \sigma)$, then $(\mu, \sigma)$ is reducible with respect to the  coordinates $x^{r+1},...,x^n$.
\end{remark}

\begin{definition}\label{definition_trinagular2}
An SDE $(\mu,\sigma)$ is triangular (with respect to the variables $(x^{r+1},...,x^n)$) if $(\mu,\sigma)$ is reducible with respect to $(x^{r+1},...,x^n)$, and furthermore for any $j\leq r$ and $\alpha=1,...,m$, the functions $\mu^j,\sigma^j_{\alpha}$ depend only on $(x^{j+1},...,x^n)$.
\end{definition}

\begin{remark}\label{remark_trinagular2}
If $(X,W)$ is a solution to an SDE which is triangular with respect to the variables $(x^{r+1},...,x^n)$, the process $X=(X^1_t,...,X^r_t)$ satisfies a triangular system of equations
\[dX^i_t=\mu^i(X^{i+1}_t,...,X^n_t)dt+\sum_{\alpha=1}^m\sigma^i_{\alpha}(X^{i+1}_t,...,X^n_t)dW^{\alpha}_t,
 \]
 for $i=1,...,r$. This means that the stochastic process $(X^1_t,...,X^r_t)$ can be recovered from the reduced process $(X',W)$ (defined in Remark \ref{remark_trinagular5}) using only Riemann and It\^o integrations and composition with smooth functions.
\end{remark}

\begin{remark}\label{remark reduction}
If $Y_1,...,Y_r$ are a set of vector fields in canonical form which are also strong symmetries of the  SDE $(\mu, \sigma)$, then, by Remark \ref{remarkreduction1}, $(\mu, \sigma)$ is in triangular form.
\end{remark}

\begin{theorem}
\label{theor:tha}
Let $K=\spann_{\mathbb{R}}\{V_1,\dots,V_k\}$ be a Lie algebra of general infinitesimal stochastic transformations  and let $x_0\in{M}$ be such that $Y_1(x_0),\dots,Y_k(x_0)$ are linearly independent, where $V_i=(Y_i,C_i,\tau_i,H_i)$. Then, there exist an open neighborhood $U$ of $x_0$ and a stochastic transformation of the form $T=(id_U,B,\eta,h)$ such that $T_*(V_1),\dots,T_*(V_k)$ are strong infinitesimal stochastic transformations.
Furthermore, the smooth functions $B$,$\eta$ and $h$ are solutions to the equations
\begin{align*}
Y_i(B)&=-B\cdot{C_i}, \notag \\
Y_i(\eta)&=-\tau_i\eta, \notag \\
Y_i(h)&=\Bigr(-\frac{\tau_i}{2}+C_i\Bigl)h-H_i, \notag
\end{align*}
for $i=1,\dots,k$.
\end{theorem}

\begin{proof}
Given the transformation $T=(id_U,B,\eta,h)$, by definition of push forward  we have
\begin{align}
T_*(V_i)=&\Bigr(Y_i,Y_i(B)\cdot{B^{-1}}+B\cdot{C_i}\cdot{B^{-1}},\tau_i+Y_i(\eta)\eta^{-1}, \notag \\
&-\frac{B}{\sqrt{\eta}}\Bigr(-\frac{\tau_i}{2}+C_i\Bigl)h+\frac{B}{\sqrt{\eta}}\cdot{H_i}+\frac{B}{\sqrt{\eta}}Y_i(h)\Bigl).\notag
\end{align}
Therefore, $T_*(V_i)$ is a strong infinitesimal stochastic transformation if and only if
\begin{align}
Y_i(B)\cdot{B^{-1}}+B\cdot{C_i}\cdot{B^{-1}}&=0, \label{equat:eqc}\\
\tau_i+Y_i(\eta)\eta^{-1}&=0, \label{equat:eqd}\\
Y_i(h)+\frac 12 \tau_i h-C_i h+H_i&=0.\label{equat:eqe}
%-\frac{B}{\sqrt{\eta}}\Bigr(-\frac{\tau_i}{2}+C_i\Bigl)h+\frac{B}{\sqrt{\eta}}\cdot{H_i}+\frac{B}{\sqrt{\eta}}Y_i(h)&=0.
\end{align}

Denote by $L_i$, $N_i$ and $Q_i$ the linear operators on $Mat(m,m)$-valued, $\mathbb{R}_+$-valued and $\mathbb{R}^m$-valued smooth functions, respectively, such that
\begin{align*}
L_i(B)&=Y_i(B)+B\cdot{C_i}=(Y_i+R_{C_i})(B), \notag \\
N_i(\eta)&=Y_i(\eta)+\eta\tau_i=(Y_i+R_{\tau_i})(\eta), \notag \\
Q_i(h)&=Y_i(h)-\Bigr(-\frac{\tau_i}{2}+C_i\Bigl)h=(Y_i-R_{-\frac{\tau_i}{2}+C_i})(h), \notag
\end{align*}
where $R_{(\cdot)}$ is the operator of right multiplication.
Equations \eqref{equat:eqc}, \eqref{equat:eqd} and \eqref{equat:eqe} are respectively
\begin{align*}
L_i(B)&=0, \notag \\
N_i(\eta)&=0, \notag \\
Q_i(h)&=-H_i. \notag
\end{align*}

As proved in \cite{DMU2}, a sufficient condition for the existence of a non-trivial solution to equations \eqref{equat:eqc} and \eqref{equat:eqd} is that there exist some real constants $c_{i,j}^k$ and $d_{i,j}^k$ such that
\begin{align}
[L_i,L_j]&=\sum_kc^k_{i,j}L_k, \\
[N_i,N_j]&=\sum_kd^k_{i,j}N_k.
\end{align}
In order to solve the last equation (which is affine) we have to prove again that there exist some real constants $e_{i,j}^k$ such that
\begin{equation}
\label{equat:eqf}
[Q_i,Q_j]=\sum_ke^k_{i,j}Q_k
\end{equation}
and that the condition
\begin{equation}
\label{equat:eqg}
-Q_i(H_j)+Q_j(H_i)=\sum_ke^k_{i,j}H_k
\end{equation}
is satisfied.
Since $K$ is a Lie algebra,  there exist some constants $e_{i,j}^k$ such that
\begin{align}
[V_i,V_j]=&\Bigr([Y_i,Y_j],Y_i(C_j)-Y_j(C_i)-\{C_i,C_j\},Y_i(\tau_j)-Y_j(\tau_i),Y_i(H_j)-Y_j(H_i) \notag \\
&-\Bigr(-\frac{\tau_i}{2}+C_i\Bigl)\cdot{H_j}+\Bigr(-\frac{\tau_j}{2}+C_j\Bigl)\cdot{H_i}\Bigl)= \notag \\
\label{equat:eqh}
=&\Bigr(\sum_ke^k_{i,j}Y_k,\sum_ke^k_{i,j}C_k,\sum_ke^k_{i,j}\tau_k,\sum_ke^k_{i,j}H_k\Bigl)
\end{align}
Equations \eqref{equat:eqf} and \eqref{equat:eqg} are proven using equation \eqref{equat:eqh} and the definition of $Q_i$.
\end{proof}

\begin{corollary}\label{triangular_form}
Let $V_1=(Y_1,C_1,\eta_1, H_1),...,V_r=(Y_r,C_r,\eta_r, H_r) $ be a solvable Lie algebra of symmetries of the SDE $(\mu, \sigma)$ such that $Y_1,...,Y_r$ are regular vector fields. Then, for any $x_0 \in M$, there exist a neighborhood $U$ of $x_0$ and a stochastic transformation $T=(\Phi,B, \eta, h)$ such that $E_T(\mu, \sigma)$ is in triangular form.
\end{corollary}

\begin{proof}
We prove that there exists a $T=T_1 \circ T_2$, where $T_1=(\Phi_1,I,1,0)$ and $T_2=(id_M,B,\eta, h)$, satisfying the thesis of the theorem. Owing to  Theorem \ref{theor:tha}, the transformation $T_2$ can be chosen  such that $T_2^*(V_1)=(Y_1,0,0,0),...,T_2^*(V_r)=(Y_r,0,0,0)$. This means that $Y_1,...,Y_r$ form a regular solvable Lie algebra of strong symmetries of $E_{T_2}(\mu, \sigma)$. By Theorem \ref{theorem_solvable_coordinate} there exists a (locally defined) map $\Phi_1$ such that $\Phi_1^*(Y_1),...,\Phi^*_1(Y_r)$ are in canonical form and they are symmetries of $E_{T_1}(E_{T_2}(\mu, \sigma)$. By Remark \ref{remark reduction}, this implies that $E_{T_1}(E_{T_2}(\mu, \sigma))$ is in triangular form. Since $E_T(\mu, \sigma)=E_{T_1}(E_{T_2}(\mu, \sigma))$ the theorem is proved.
${}\hfill$
\end{proof}

Since quasi Doob symmetries play a crucial role in the reduction and reconstruction process for many interesting SDEs, in the following we prove the analogous of Theorem \ref{theor:tha} when considering only quasi Doob transformations.

\begin{theorem}\label{theoRaddrQD}
Let $K=span\{V_1,\dots,V_k\}$ be a Lie algebra of quasi Doob symmetries for an SDE $(\mu, \sigma)$ and let $x_0\in{M}$ be such that $Y_1(x_0),\dots,Y_k(x_0)$ are linearly independent, where $V_i=(Y_i,C_i,\tau_i,\sigma^T \cdot \nabla k_i)$. Then, there exist an open neighborhood $U$ of $x_0$ and a stochastic transformation of the form $T=(id_U,B,\eta,\sigma^T\cdot \nabla \bar k)$ such that $T_*(V_1),\dots,T_*(V_k)$ are strong infinitesimal stochastic symmetries for $(\mu, \sigma)$.

Furthermore the smooth functions $B$,$\eta$ and $k$ are solutions to the equations
\begin{align}
Y_i(B)&=-B\cdot{C_i}, \notag \\
Y_i(\eta)&=-\tau_i\eta, \notag \\
Y_i(\bar k)&=-k_i, \notag
\end{align}
for $i=1,\dots,k$.
\end{theorem}
\begin{proof}
The proof follows the same line of the proof of Theorem \ref{theor:tha}, except for equation \eqref{equat:eqe} that becomes
\begin{equation*}
Y_i(\sigma^T\cdot \nabla (\bar k))+\frac 12 \tau_i \sigma^T\cdot \nabla (\bar k)-C_i \cdot \sigma^T \cdot \nabla( \bar k )+\sigma^T\cdot \nabla (k_i)=0
\end{equation*}
Therefore, using Lemma \ref{lemma:lea} we get
\begin{align*}
& Y_i(\sigma^T)\cdot \nabla( \bar k)+\sigma^T\cdot Y_i(\nabla ( \bar k ))+\frac 12 \tau_i \sigma^T\cdot \nabla (\bar k)-C_i \cdot \sigma^T \cdot \nabla (\bar k )+\sigma^T\cdot \nabla ( k_i)=\\
&\sigma^T\left((\nabla Y_i)^T\cdot \nabla (\bar k)+Y_i(\nabla (\bar k)) +\nabla ( k_i) \right)=0
\end{align*}
Moreover, since
\begin{equation*}
(\nabla Y_i)^T\cdot \nabla (\bar k)+Y_i(\nabla (\bar k)) =\nabla (Y_i(\bar k)),
\end{equation*}
we can rewrite the previous equation as
\begin{equation*}
\sigma^T\nabla \left( Y_i(\bar k)+k_i\right)=0.
\end{equation*}
Since  $V_i$ form a Lie algebra, this concludes the proof.
\end{proof}

\subsection{Reconstruction of SDEs via general symmetries}

We recall the meaning of the reconstruction procedure of a process starting from a reduced one in a stochastic framework (\cite{DMU2}).

\begin{definition}\label{standard_reconstruction}
	Let $O$ and $O^\prime$ be two $(\Omega, \mathcal{F}, (\mathcal{F}_t)_{t \in [0,\mathcal{T}])}, \mathbb{P})$ processes on $M$ and $M^\prime$, respectively. We say that $O$ can be reconstructed from $O^\prime$ if there exists a smooth function $F: \mathbb{R}^{(m+1)}\times M^\prime \times M \rightarrow M$ and a stochastic absolutely continuous process $\alpha_t$ such that
	\begin{equation}
	O_t=F\left(\int_{0}^{\alpha_t}f_0(s,O^\prime_s)ds,\int_{0}^{\alpha_t}f_1(s,O^\prime_s)dW^1_s,\dots,\int_{0}^{\alpha_t}f_m(s,O^\prime_s)dW^m_s,O^\prime_{\alpha_t},O_0\right)
	\end{equation}
	where $W^1, \dots, W^m$ are $m$ Brownian motions and $f_i:\mathbb{R}\times M^\prime \rightarrow \mathbb{R}$ are smooth functions.
\end{definition}
\begin{remark}\label{triangular_integration}
Let $(X,W)$ be a solution to the symmetric SDE $(\mu,\sigma)$ admitting a solvable $r$-dimensional Lie algebra of infinitesimal symmetries. Then, by Corollary \ref{triangular_form}, there exists a (local) stochastic transformation $T$ transforming the SDE $(\mu,\sigma)$ into a new SDE $E_T(\mu,\sigma)$ of triangular form (see Definition \ref{definition_trinagular2} and Remark \ref{remark_trinagular2}). This means that the process $X'=P_T(X)=(X'^1,...,X'^n)$ can be reconstructed from the reduced process $O'=(X'^{r+1},...,X'^n)$ (which is the projection of $X'$ on the reduced space $\mathbb{R}^{n-r}$) since $(X',P_T(W))$ solves the triangular SDE $(\mu',\sigma')$. This implies that $X$ can be reconstructed from $O'$ too, since $X=P_{T^{-1}}(X')$ and the operations involved in the computation of $P_{T^{-1}}(X')$ are the ones used in Definition \ref{standard_reconstruction}. This property of symmetric SDEs is better explained in the examples discussed in Section \ref{section 5}. For a more detailed study of this problem see also \cite{AlDeMoUg2019,DMU2,DeUg2017}.
\end{remark}
On the other hand, when we consider reductions of SDEs  by means of general symmetries, including a random change of the underlying probability measure, there is no hope to recover the explicit form of the original process in terms of the reduced one in the sense of the previous definition. In this Section we discuss what we can reasonably obtain for the solution to the initial SDE throughout a Lie's symmetries analysis involving a measure change and we start by giving the following new definition of reconstruction.

\begin{definition}\label{definition_integrable}
	Let $O$ be a $(\Omega, \mathcal{F}, (\mathcal{F}_t)_{t \in [0,\mathcal{T}])} , \mathbb{P})$ process on $M$ and $O^\prime$ be a $(\Omega, \mathcal{F}, (\mathcal{F}_t)_{t \in [0,\mathcal{T}])} , \mathbb{Q})$ process on  $M^\prime$. We say that $O$ can be reconstructed from $O^\prime$ if there exist two smooth functions $F_i:\mathbb{R} \times \mathbb{R}^{(m+1)}\times M^\prime \times M \rightarrow M,(i=1,2$) and an adapted absolutely continuous process $\alpha_t$, such that, for any continuous and bounded function $g:M^k \rightarrow \mathbb{R}$ and for every choice of times $t_1,t_2,\dots,t_k\in [0,\mathcal{T}]$,  we have that
	\begin{equation}
\mathbb{E}_{\mathbb{P}}[g(O_{t_1}, O_{t_2},\dots,O_{t_k})]	=\mathbb{E}_{\mathbb{Q}}[g(F_1^{\alpha_{t_1}},F_1^{\alpha_{t_2}},\dots,F_1^{\alpha_{t_k}})\exp(	F_2^{\alpha_{t_{\max}}})]
	\end{equation}
	where
	\begin{equation}
	F_1^{t}:=F_1\left(t,\int_{0}^{t}f_0(s,O^\prime_s)ds,\int_{0}^{t}f_1(s,O^\prime_s)dW^1_s,\dots,\int_{0}^{t}f_m(s,O^\prime_s)dW^m_s,O^\prime_t,O_0\right),
	\end{equation}
	and
	\begin{multline}
	F_2^{t_{\max}}:=\\
	F_2\left(t,\int_{0}^{{t}_{\max}}f_0(s,O^\prime_s)ds,\int_{0}^{{t}_{\max}}f_1(s,O^\prime_s)dW^1_s,\dots,\int_{0}^{{t}_{\max}}f_m(s,O^\prime_s)dW^m_s,O^\prime_{{t}_{\max}},O_0\right),
	\end{multline}
	where ${t}_{\max}=\max(t_1,\dots,t_k)$,
	$W^1, \dots, W^m$ are $m$ Brownian motions and $f_k:\mathbb{R}\times M^\prime \rightarrow \mathbb{R}$  are smooth functions for any $k=0,1,\dots, m$.
\end{definition}

\begin{remark}
	The above definition says that the $\mathbb{P}$-finite dimensional distributions (f.d.d) of the process $O_t$ can be expressed in terms of (slightly modified) $\mathbb{Q}$- f.d.d. of the process $O^\prime_t$. The modified probability law is absolutely continuous with respect to $\mathbb{Q}$ with a Radon-Nikodim derivative which is Markovian with respect to the natural filtration generated by the process $O'_t$.
\end{remark}

The reconstruction result given in the following theorem holds for general stochastic symmetries, i.e. takes into account  all possible kinds of random transformations, including random time changes. Anyway, in the examples of Section \ref{section 5} we privilege general transformations without time changes, that we call Girsanov transformations.

\begin{theorem}[Reconstruction theorem]\label{reconstruction}
	Let $(X,W)$ be a weak  solution to an SDE $(\mu,\sigma)$ and let $(X^\prime,W^\prime)$ be a weak solution to the reduced SDE, obtained by Corollary \ref{triangular_form}. Then $(X,W)$ can be reconstructed from $(O^\prime,W^{\prime})$, where $O^\prime=(X^{\prime^{2}},\dots,X^{\prime^{n}})$. In particular, for any continuous and bounded function $g:M \rightarrow \mathbb{R}$, we have
\begin{align}
		\mathbb{E}_{\mathbb{P}}[g(X_t)]:=&\mathbb{E}_{\mathbb{Q}}\left[g\left (F^t_1\left ( \int_{0}^{t}f_0(s,O^\prime_s)ds,\int_{0}^{t}f_1(s,O^\prime_s)dW^{{\prime}^{1}}_s,\right.\right.\right.\nonumber\\
		&\left.\left.\dots,\int_{0}^{t}f_m(s,O^\prime_s)dW^{{\prime}^{m}}_s,O^\prime_t,X_0 \right )\right ) \exp\left( F^t_2\left(\int_{0}^{t}f_0(s,O^\prime_s)ds,\right.\right.\nonumber\\
		&\left.\left.\left.\int_{0}^{t}f_1(s,O^\prime_s)dW^{{\prime}^{1}}_s,\dots,\int_{0}^{t}f_m(s,O^\prime_s)dW^{{\prime}^{m}}_s,O^\prime_t,X_0 \right)\right)\right].
\end{align}
\end{theorem}

\begin{proof}
The two solution processes $(X,W)$ and $(X^\prime,W^\prime)$ are related by the transformation $T$ given in Corollary \ref{triangular_form}, which reduces the original SDE to an SDE in triangular form. Since the action of the stochastic transformation $T$ on the process is given in Definition \ref{defi:dea}, using  the same notations, we can write, $\forall t \in [0,\mathcal{T}]$,
\begin{equation}
\mathbb{E}_{\mathbb{P}}[g(X_t)]=\mathbb{E}_{\mathbb{P}}[g(P_{T^{-1}}(X^\prime_t))]=
\mathbb{E}_{\mathbb{Q}}\left[g(P_{T^{-1}}(X^\prime_t)) \frac{d\mathbb{P}}{d\mathbb{Q}}\Bigg|_{\mathcal{F}_{t}}\right]
\end{equation}	
where a measure change in the expectation has been performed, $T^{-1}$ is the inverse transformation given in \eqref{inversetransformation}
and
\begin{equation}
	\frac{d\mathbb{P}}{d\mathbb{Q}}\Bigg|_{\mathcal{F}_{t}} =  \exp \left( -\int_0^t h^T(X_s)dW_s+\frac{1}{2}\int_0^t h^Th(X_s)ds \right).
\end{equation}
In order to express the integrands in the expectation with respect to $\mathbb{Q}$ in terms of the reduced process $(X^\prime,W^\prime)$, by Definition \ref{defi:dea} we get
\[
dW_t=\frac{1}{\sqrt{\eta(X_t)}}B^{-1}(X_t)d\tilde{W}_t+h(X_t)dt
\]
so that
\begin{multline*}
	\frac{d\mathbb{P}}{d\mathbb{Q}}\Bigg|_{\mathcal{F}_{\alpha_t}} =  \exp \left( -\int_0^{\alpha_t} h^T(P_{T^{-1}}(X^\prime_s))\frac{1}{\sqrt{\eta(P_{T^{-1}}(X^\prime_s))}}B^{-1}(P_{T^{-1}}(X^\prime_s))d\tilde{W}_s\right)\times\\
	\times \exp \left(-\frac{1}{2}\int_0^{\alpha_t} (h(P_{T^{-1}}(X^\prime_s)))^2ds \right),
\end{multline*}
where $\alpha_t$ is the inverse random time change as introduced in \eqref{alpha} and $W^\prime_t=H_{\eta}(\tilde{W}_t)$ is a $\mathbb{Q}-$ Brownian motion.
Since the reduced process $(X^\prime,W^\prime)$ is by construction of triangular form it can be integrated (see Definition \ref{remark_trinagular2}, Remark \ref{remark_trinagular2} and Remark \ref{triangular_integration}).
\end{proof}

\begin{remark}
In Theorem \ref{reconstruction}, for sake of simplicity,  we consider the case of functions $g(O_t)$ depending only on the process at time $t\in [0,\mathcal{T}]$, but the proof can be easily generalized to the case of continuous bounded functions $g:M^k \rightarrow \mathbb{R}$  taking the expectations of the form $g(O_{t_1},\cdots,O_{t_k})$ (for any $t_1,...,t_k\in [0,\mathcal{T}]$) as required by Definition \ref{definition_integrable}.
\end{remark}

\section{Examples}\label{section 5}

In order to include in our setting also time dependent transformations, in the following examples we add to the original SDE a further component admitting solution $Z_t=t-t_0$.

\subsection{Bessel process}
Let us consider the SDE associated with the well-known Bessel equation
\[
\begin{pmatrix}
dX_t \\
dZ_t
\end{pmatrix}
=
\begin{pmatrix}
\frac {a}{X_t} \\
1
\end{pmatrix}
dt+
\begin{pmatrix}
1 \\
0
\end{pmatrix}
dW_t,
\]
where $W_t$ is a one dimensional Brownian motion and $a\not= 0$ is a real constant (for a recent review on one-dimensional Bessel process see e.g. \cite{Carr,Lawler}).
If we consider the Lie algebra generated by the two standard symmetries  with random time change (but without the change of the reference measure)
\[
V_1=\left(
\begin{pmatrix}
0 \\
1
\end{pmatrix}
,0,0,0\right)
\]

\[
V_2=\Biggr(
\begin{pmatrix}
\frac{x}{2}\\
z
\end{pmatrix}
,0,1,0\Biggl),
\]
we can look for the stochastic transformation  $T=(\Phi, B, \eta,0 )$ transforming $V_1$ and $V_2$ in strong symmetries in canonical form. In particular, from Theorem \ref{theor:tha}, solving $Y_i(B)=-BC_i$ ($i=1,2$), i.e.
\[ B_z=0, \quad \frac 12 xB_x+zB_z=0\]
we find that $B$ has to be constant and solving $Y_i(\eta)=-\tau_i\eta$ ($i=1,2$), i.e.
\[ \eta_z=0, \quad \frac 12 x\eta_x+z\eta_z=-\eta\]
we get $\eta =\frac {C}{x^2}$, where $C\in \mathbb{R}$.
Moreover, in order to put $Y_1$ and $Y_2$ in canonical form, we consider the flows $\Phi^i_{a_i}$ of $Y_i$ ($i=1,2$) given by
\[
\Phi_{a_1}^1(x,z)=
\begin{pmatrix}
x\\
z+a_1
\end{pmatrix}
\]
\[
\Phi_{a_2}^2(x,z)=
\begin{pmatrix}
xe^{\frac 12 a_2} \\
ze^{a_2}
\end{pmatrix}.
\]
Hence, if we consider the point $p=(1,1)^T$ and the functions $F:\mathbb{R}^2\to M$ given by
\[
F(a_1,a_2)= \Phi_{a_1}^1(\Phi_{a_2}^2(p))= \begin{pmatrix}
e^{\frac 12 a_2}\\
e^{a_2}+a_1
\end{pmatrix}
\]
the function $\Phi$ can be obtained as the inverse of $F$ and is given by
\[
\Phi(x,z)=
\begin{pmatrix}
z-x^2\\
2\log(x)
\end{pmatrix}.
\]
Applying the previous stochastic transformation (with $B=1$, $\eta=\frac{1}{x^2}$) we can reduce Bessel SDE, and we get the following integrable SDE
\[
\begin{pmatrix}
dX^\prime_t \\
dZ^\prime_t
\end{pmatrix}
=
\begin{pmatrix}
-2a\exp(Z^\prime_t) \\
(2a-1)
\end{pmatrix}
dt+
\begin{pmatrix}
- 2\exp(Z^\prime_t)\\
2
\end{pmatrix}
dW^\prime_t.
\]
Indeed, since it is in triangular form, we obtain
\[
Z^\prime_t=Z^\prime_0+(2a-1)t+2W^\prime_t
\]
\[
X^\prime_t=Z^\prime_0-\int_0^t{[2a\exp(Z^\prime_s)ds+2\exp(Z^\prime_s)dW^\prime_s]}
\]
In this case, since there is no change of the measure, we can reconstruct the original process starting from the reduced one following the procedure illustrated in \cite{DMU2}. We have to apply the following inverse transformation:
$$T^{-1}=(\Phi^{-1}, 1, (\eta\cdot\Phi^{-1})^{-1},0) $$
Recalling that the action of the stochastic transformation $T^{-1} $ on the process $X^\prime$ is given by
\[
P_{T^{-1}}(X^\prime)=\Phi^{-1}(H_\alpha(X^\prime))
\]
we find
\[
X_t=\exp(Z^\prime_{{\alpha_t}}).
\]
This formula shows that a Bessel process  can be seen as a time changed geometric Brownian motion. So we recover the well-known Lamperti's Theorem which relates the Bessel process with a time changed geometric Brownian motion (see, e.g. \cite{Carr}) by using a Lie's symmetry approach.\\

The time change
$$\alpha_t=\int_{0}^{t}X^2_sds=\int_{0}^{t}\exp(Z^\prime_{{0}}+(2a-1)s+2W^\prime_s)ds$$
is then the Yor's process, which is very useful in financial applications (see \cite{Carr}). \\

On the other hand, if we look for quasi Doob symmetries for the previous SDE, we find
\[
V_1=\Biggr(
\begin{pmatrix}
-1 \\
0
\end{pmatrix}
,0,0,\frac {a}{x^2}\Biggl)
\]

\[
V_2=\Biggr(
\begin{pmatrix}
0\\
1
\end{pmatrix}
,0,0,0\Biggl).
\]

Since $V_1$ and $V_2$ form an abelian solvable Lie algebra, we can apply Theorem \ref{theoRaddrQD} in order to find  the quasi Doob stochastic transformation $T=(\Phi, B, \eta, h)$  transforming Bessel SDE in a new SDE in triangular form admitting only strong symmetries. In particular, in order to find $\Phi$, we start by computing
 the flows of $Y_1$ and $Y_2$:
\[
\Phi_{a_1}^1(x,z)=
\begin{pmatrix}
x-a_1\\
z
\end{pmatrix}
\]
\[
\Phi_{a_2}^2(x,z)=
\begin{pmatrix}
x \\
z+a_2
\end{pmatrix}.
\]
If we consider the point $p=(0,0)^T$ and the function $F\colon\mathbb{R}^2\to M$ given by
\[
F(a_1,a_2)=\Phi_{a_1}^1(\Phi_{a_2}^2(p))=
\begin{pmatrix}
- a_1  \\
a_2
\end{pmatrix},
\]
the function $\Phi\colon M\to\mathbb{R}^2$, which is the inverse of $F$, is given by
\[
\Phi(x,z)=
\begin{pmatrix}
-x\\
z
\end{pmatrix}.
\]
Now, it easy to check that
\[
\Phi_*(Y_1)=\nabla(\Phi)\cdot Y_1=
\begin{pmatrix}
1 \\
0
\end{pmatrix}
=:Y'_1,
\]
\[
\Phi_*(Y_2)=\nabla(\Phi)\cdot Y_2=
\begin{pmatrix}
2k\sqrt{x}e^{kz} \\
1
\end{pmatrix}
=:Y'_2.
\]
Next, by Theorem \ref{theor:tha}, the other components of $T$ are $B=1$, $\eta=1$
since $C_i$ and $\tau_i$ vanishes.
Moreover, the equations for $h$ are
\begin{align}
h_x&=\frac {a}{x^2}, \notag \\
h_z&=0. \notag
\end{align}
Taking  $h=-\frac ax$ we can compute
the transformed SDE $(\mu',\sigma'):=E_T(\mu,\sigma)$ as
\[
\mu'=\Biggr(\frac{1}{\eta}[L(\Phi)+\nabla(\Phi)\cdot\sigma\cdot h]\Biggl)\circ\Phi^{-1}=
\begin{pmatrix}
0 \\
1
\end{pmatrix}
\]
\[
\sigma'=\Biggr(\frac{1}{\sqrt{\eta}}[\nabla(\Phi)\cdot\sigma\cdot B^{-1}]\Biggl)\circ\Phi^{-1}=
\begin{pmatrix}
-1 \\
0
\end{pmatrix}
\]

The reduced SDE takes the very simple form
\[
\begin{pmatrix}
dX^\prime_t \\
dZ^\prime_t
\end{pmatrix}
=
\begin{pmatrix}
0 \\
1
\end{pmatrix}
dt+
\begin{pmatrix}
-1 \\
0
\end{pmatrix}
dW^\prime_t,
\]

Since the function $\Phi\colon M\to\mathbb{R}^2$ is given by
\[
\Phi^{-1}(x,z)=
\begin{pmatrix}
-x^\prime\\
z^\prime
\end{pmatrix}
\]
and there is no time change (i.e. $\eta=1$), following the proof of Theorem \ref{reconstruction} we obtain

\begin{equation}
	\mathbb{E}_{\mathbb{P}}[g(X_t)]=\mathbb{E}_{\mathbb{Q}}\left[g(\Phi^{-1}(X^\prime_t))	\frac{d\mathbb{P}}{d\mathbb{Q}}\Bigg|_{\mathcal{F}_{t}} \right] 	
\end{equation}
with
\begin{equation}
	\frac{d\mathbb{P}}{d\mathbb{Q}}\Bigg|_{\mathcal{F}_{t}} =  \exp \left( -\int_0^t h(X_s)dW_s+\frac{1}{2}\int_0^t (h(X_s))^2ds \right).
\end{equation}
Substituting $dW_t=dW^\prime_t+h({X_t})dt$ we have
\begin{equation}\label{RN_derivative}
\frac{d\mathbb{P}}{d\mathbb{Q}}\Bigg|_{\mathcal{F}_{t}} =  \exp \left( -\int_0^t h(X_s)dW^\prime_s-\frac{1}{2}\int_0^t (h(X_s))^2ds \right),
\end{equation}
and, by definition of $h$,
\begin{equation}
	\frac{d\mathbb{P}}{d\mathbb{Q}}\Bigg|_{\mathcal{F}_{t}} =  \exp \left( -\int_0^t \frac{a}{X^\prime_s}dW^\prime_s-\frac{1}{2}\int_0^t \frac{a^2}{(X^\prime_s)^2}ds \right).
\end{equation}
Since the reduced SDE is integrable
\begin{equation}
\frac{d\mathbb{P}}{d\mathbb{Q}}\Bigg|_{\mathcal{F}_{t}} =  \exp \left( -\int_0^t \frac{a}{X^\prime_0-W^\prime_s}dW^\prime_s-\frac{1}{2}\int_0^t \frac{a^2}{(X^\prime_0-W^\prime_s)^2}ds \right).
\end{equation}
Finally by assuming $X_0=X^\prime_0$ we get
\begin{multline}
\mathbb{E}_{\mathbb{P}}[g(X_t)]=\\
\mathbb{E}_{\mathbb{Q}}\left[g(-X_0+W^\prime_t)  \exp \left( -\int_0^t \frac{a}{(X_0-W^\prime_s)}dW^\prime_s-\frac{1}{2}\int_0^t \frac{a^2}{(X_0-W^\prime_s)^2}ds \right)\right]. 	
\end{multline}
We remark that $W^\prime_t$ is a $\mathbb{Q}-$ Brownian motion. This first example shows clearly in which sense we can reconstruct our original Bessel process starting from the reduced SDE which is simply a Brownian motion starting from $-X_0$. The $\mathbb{P}$-f.d.d. of the Bessel process admit a representation in terms of those of a Brownian motion starting by $-X_0$ through a functional similar to Feynman-Kac formula. Moreover, it is possible to explicitely compute this functional, depending essentially on a Brownian motion. 

\subsection{CIR model}\label{section_CIR}
Let us consider the Cox-Ingersoll-Ross (CIR) model
\[
\begin{pmatrix}
dX_t \\
dZ_t
\end{pmatrix}
=
\begin{pmatrix}
aX_t+b \\
1
\end{pmatrix}
dt+
\begin{pmatrix}
\sigma_0\sqrt{X_t} \\
0
\end{pmatrix}
dW_t,
\]
where $W_t$ is a one dimensional Brownian motion and $a$, $b$, $\sigma_0$ are constants, which is widely used in mathematical finance to describe the behavior of the interest rates (see \cite{Brigo2006} Chapter 3). An analysis of Lie's point symmetries of the Kolmogorov equation associated with CIR model is presented in \cite{Craddock2004}.
We can compute the following one parameter family of quasi Doob infinitesimal symmetries of this model ($k\in \mathbb{R}$)
\[
V_1=\Biggr(
\begin{pmatrix}
e^{-kz}\sqrt{x} \\
0
\end{pmatrix}
,0,0,e^{-kz}\left[ \frac{a+2k}{2\sigma_0}+\frac 1x \left( \frac {\sigma_0^2-4b}{8\sigma_0}\right)\right]\Biggl)
\]

\[
V_2=\Biggr(
\begin{pmatrix}
0\\
1
\end{pmatrix}
,0,0,0\Biggl).
\]

Since $V_1$ and $V_2$ form a solvable Lie algebra of quasi Doob symmetries for the SDE, we can apply Theorem \ref{theoRaddrQD} in order to find the quasi Doob stochastic transformation $T=(\Phi, B, \eta, h)$  transforming the CIR model in a new SDE in triangular form. In particular we compute a stochastic transformation $T=(\Phi,B,\eta,h)$ transforming $V_1$ and $V_2$ into strong symmetries $V'_1$ and $V'_2$ for the transformed SDE $E_T(\mu,\sigma)$ such that the vector fields $Y'_1$ and $Y'_2$ are in canonical form. As in the previous example, in order to find $\Phi$, we  compute the flows of $Y_1$ and $Y_2$
\[
\Phi_{a_1}^1(x,z)=
\begin{pmatrix}
(x+\frac 12 a_1e^{-kz})^2 \\
z
\end{pmatrix}
\]
\[
\Phi_{a_2}^2(x,z)=
\begin{pmatrix}
x \\
z+a_2
\end{pmatrix}
\]
and we consider the point $p=(0,0)^T$ and the function $F\colon\mathbb{R}^2\to M$ given by
\[
F(a_1,a_2)=\Phi_{a_1}^1(\Phi_{a_2}^2(p))=
\begin{pmatrix}
\left( \frac 12 a_1 e^{-ka_2}\right)^2 \\
a_2
\end{pmatrix}.
\]
The function $\Phi\colon M\to\mathbb{R}^2$, which is the inverse of $F$, is given by
\[
\Phi(x,z)=
\begin{pmatrix}
2\sqrt{x}e^{kz} \\
z
\end{pmatrix}.
\]
Now, if we compute $\Phi_*(Y_1)$ and $\Phi_*(Y_2)$, we get
\[
\Phi_*(Y_1)=\nabla(\Phi)\cdot Y_1=
\begin{pmatrix}
\frac{e^{kz}}{\sqrt{x}} & 2k\sqrt{x} e^{kz} \\
0 & 1
\end{pmatrix}
\cdot
\begin{pmatrix}
e^{-kz}\sqrt{x} \\
0
\end{pmatrix}
=
\begin{pmatrix}
1 \\
0
\end{pmatrix}
=:Y'_1,
\]
\[
\Phi_*(Y_2)=\nabla(\Phi)\cdot Y_2=
\begin{pmatrix}
\frac{e^{kz}}{\sqrt{x}} & 2k\sqrt{x} e^{kz} \\
0 & 1
\end{pmatrix}
\cdot
\begin{pmatrix}
0\\
1
\end{pmatrix}
=
\begin{pmatrix}
2k\sqrt{x}e^{kz} \\
1
\end{pmatrix}
=
\begin{pmatrix}
kX' \\
1
\end{pmatrix}
=:Y'_2,
\]
that are in canonical form.

Next, by Theorem \ref{theor:tha}, the second and the third components of $T$ are $B=1$ and $\eta=1$
since $C_i$ and $\tau_i$ are zero.
Remember that in this example $B$ is not a matrix-valued function, because the Brownian motion has dimension one.

Moreover, the equations for $h$ are
\begin{align}
Y_1(h)&=-e^{-kz} \left[ \frac{a+2k}{2\sigma_0}+\frac 1x \left( \frac {\sigma_0^2-4b}{8\sigma_0}\right)\right], \notag \\
Y_2(h)&=0, \notag
\end{align}
then
\begin{align}
h_x&=-\frac{a+2k}{2\sigma_0\sqrt x}- \frac {\sigma_0^2-4b}{8\sigma_0x\sqrt x}, \notag \\
h_z&=0. \notag
\end{align}

We obtain $h(x,z)=-\frac{a+2k}{\sigma_0}\sqrt x+ \frac {\sigma_0^2-4b}{4\sigma_0\sqrt x}+C$, and we can compute
the transformed SDE $(\mu',\sigma'):=E_T(\mu,\sigma)$ as
\[
\mu'=\Biggr(\frac{1}{\eta}[L(\Phi)+\nabla(\Phi)\cdot\sigma\cdot h]\Biggl)\circ\Phi^{-1}=
\begin{pmatrix}
0 \\
1
\end{pmatrix}
\]
\[
\sigma'=\Biggr(\frac{1}{\sqrt{\eta}}[\nabla(\Phi)\cdot\sigma\cdot B^{-1}]\Biggl)\circ\Phi^{-1}=
\begin{pmatrix}
\sigma_0e^{kZ'_t} \\
0
\end{pmatrix}
\]
Since $\sqrt{X_t}=\frac{X^\prime_t}{2}\exp(-kt)$ and using \eqref{RN_derivative}
the Radon-Nikodym derivative up to time $t$ of the measure $\mathbb{P}$ with respect to $\mathbb{Q}$ becomes

\begin{align}
\frac{d\mathbb{P}}{d\mathbb{Q}}\Bigg|_{\mathcal{F}_{t}} =  \exp \left( - \frac{a+2k}{2\sigma_0}\int_0^t {X^\prime_s}\exp(-ks)dW^\prime_s+\frac {(\sigma_0^2-4b)}{2\sigma_0}\int_0^t \frac{\exp(ks)}{X^\prime_s}dW^\prime_s \right )\times\nonumber\\
\times \exp \left(-\frac{1}{2}\int_0^t\left[-\frac{a+2k}{2\sigma_0} {X^\prime_s}\exp(-ks)
+\frac {(\sigma_0^2-4b)}{2\sigma_0} \frac{\exp(ks)}{X^\prime_s}\right]^2ds \right).
\end{align}
The reduced SDE is integrable, that is $X^\prime_s=X_0+\sigma_0 \int_{0}^{s}\exp (k\tau) dW^\prime_\tau$, so we get
\begin{align}
	\frac{d\mathbb{P}}{d\mathbb{Q}}\Bigg|_{\mathcal{F}_{t}} = & \exp \left( - \frac{a+2k}{2\sigma_0}\int_0^t \exp(-ks)\left(X_0+\sigma_0 \int_{0}^{s}\exp (k\tau) dW^\prime_\tau\right)dW^\prime_s\right. \nonumber\\
& \left.+\frac {(\sigma_0^2-4b)}{2\sigma_0}\int_0^t \frac{\exp(ks)}{(X_0+\sigma_0 \int_{0}^{s}\exp k\tau dW^\prime_\tau)}dW^\prime_s \right )\times \nonumber\\
	&\times \exp \left (-\frac{1}{2}\int_0^t\left[-\frac{a+2k}{2\sigma_0} \exp(-ks)\left(X_0+\sigma_0 \int_{0}^{s}\exp k\tau dW^\prime_\tau \right)\right.\right.\nonumber\\
&	\left.\left.+\frac {(\sigma_0^2-4b)}{2\sigma_0} \frac{\exp(ks)}{(X_0+\sigma_0 \int_{0}^{s}\exp(k\tau) dW^\prime_\tau)}\right]^2ds \right),
\end{align}
where $W^\prime_t$ is a $\mathbb{Q}-$ Brownian motion. We finally obtain

\begin{equation}
\mathbb{E}_{\mathbb{P}}[g(X_t)]=\mathbb{E}_{\mathbb{Q}}\left[g\left(\frac{\exp(-2kt)}{4}\left(X_0+\sigma_0 \int_{0}^{t}\exp (ks) dW^{\prime}_s\right)^2 \ \right) \frac{d\mathbb{P}}{d\mathbb{Q}}\Bigg|_{\mathcal{F}_{t}}\right]. 	
\end{equation}

\subsection{Ornstein-Uhlenbeck model}

Let us consider the OU model
\[
\begin{pmatrix}
dX_t \\
dZ_t
\end{pmatrix}
=
\begin{pmatrix}
aX_t+b \\
1
\end{pmatrix}
dt+
\begin{pmatrix}
1 \\
0
\end{pmatrix}
dW_t,
\]
where $W_t$ is a one dimensional Brownian motion and  $a$ and $b$ are constants.
If we look for a two-dimensional algebra of infinitesimal Doob symmetries for this model, we find
\[
V_1=\Biggr(
\begin{pmatrix}
\frac 12 e^{-az} \\
0
\end{pmatrix}
,0,0,ae^{-az}\Biggl)
\]

\[
V_2=\Biggr(
\begin{pmatrix}
0\\
1
\end{pmatrix}
,0,0,0\Biggl).
\]

Since $[V_1, V_2]=aV_1$,  $V_1$ and $V_2$ form a solvable Lie algebra of  Doob symmetries for the SDE and we can apply Theorem \ref{theoRaddrQD} in order to find the quasi Doob stochastic transformation $T=(\Phi, B, \eta, h)$  transforming the OU model in a new SDE in triangular form.

Following the same line of the previous examples, we find $B=1$, $\eta=1$, $h=-2ax+c$, with $c$ an arbitrary constant, and
\[
\Phi(x,z)=
\begin{pmatrix}
2x e^{az} \\
z-1
\end{pmatrix}
\]

If we chose $c=0$ we can compute
the transformed SDE $(\mu',\sigma'):=E_T(\mu,\sigma)$ as
\[
\mu'=\Biggr(\frac{1}{\eta}[L(\Phi)+\nabla(\Phi)\cdot\sigma\cdot h]\Biggl)\circ\Phi^{-1}=
\begin{pmatrix}
2be^{a(Z'_t+1)} \\
1
\end{pmatrix}
\]
\[
\sigma'=\Biggr(\frac{1}{\sqrt{\eta}}[\nabla(\Phi)\cdot\sigma\cdot B^{-1}]\Biggl)\circ\Phi^{-1}=
\begin{pmatrix}
2e^{a(Z'_t+1)} \\
0
\end{pmatrix}
\]
which can be easily integrated.
Using \eqref{RN_derivative} with the current expression for $h$ we have

\begin{equation}
\frac{d\mathbb{P}}{d\mathbb{Q}}\Bigg|_{\mathcal{F}_{t}} =  \exp \left( \int_0^t 2aX_sdW^\prime_s-{2}a^2\int_0^t (X_s)^2ds \right).
\end{equation}
Since
\[
\Phi^{-1}(x^\prime,z^\prime)=
\begin{pmatrix}
\frac{x^\prime}{2} e^{-a(z^\prime + 1)} \\
z^\prime+1
\end{pmatrix}
\]
by expressing the integrands in terms of $X^\prime$  we get
\[
\frac{d\mathbb{P}}{d\mathbb{Q}}\Bigg|_{\mathcal{F}_\mathcal{T}} =  \exp \left( a\int_0^t X^\prime_s\exp(-as)dW^\prime_s-\frac{a^2}{2}\int_0^t (X^\prime_s)^2\exp(-2as)ds \right).
\]
and finally
\begin{align}
\frac{d\mathbb{P}}{d\mathbb{Q}}\Bigg|_{\mathcal{F}_{t}} =  \exp \left( a\int_0^t \left[X_0+\frac{2b}{a}(\exp(as)-1)+2\int_{0}^{s}\exp(a\tau)dW^\prime_\tau\right]\exp(-as)dW^\prime_s \right) \times\nonumber  \\
\times \exp \left(-\frac{a^2}{2}\int_0^t \left(\left[X_0+\frac{2b}{a}(\exp(as)-1)+2\int_{0}^{s}\exp(a\tau)dW^\prime_\tau  \right]^2\exp(-2as)\right)ds \right).
\end{align}

On the other hand, if we chose $c=-b$, the reduced system becomes much more simple

\[
\begin{pmatrix}
dX^\prime_t \\
dZ^\prime_t
\end{pmatrix}
=
\begin{pmatrix}
0\\
1
\end{pmatrix}
dt+
\begin{pmatrix}
2\exp(a(Z_t^\prime+1)) \\
0
\end{pmatrix}
dW^\prime_t,
\]
and we obtain
\begin{equation}
\mathbb{E}_{\mathbb{P}}[g(X_t)]=\mathbb{E}_{\mathbb{Q}}\left[g(\Phi^{-1}(X^\prime_t)) \frac{d\mathbb{P}}{d\mathbb{Q}}\Bigg|_{\mathcal{F}_{t}}\right]
\end{equation}
where, using that $X^\prime_t=X_0+2\int_{0}^{t}\exp(as)dW^\prime_s$,
\begin{align}
\frac{d\mathbb{P}}{d\mathbb{Q}}\Bigg|_{\mathcal{F}_{t}}=&\exp \left(a\int_{0}^{t}\exp(-as)\left(X_0+2\int_{0}^{t}\exp(as)dW^\prime_s\right)-\frac{a^2}{2}\int_{0}^{t}\exp(-2as)(X^\prime_s)^2ds\right) \times\nonumber \\
&\times \exp\left(bW^\prime_t-\frac{b^2}{2}t\right)\left(
\exp(-ab)\int_{0}^{t}\exp(-as)X^\prime_sds\right)
\end{align}
and
\begin{equation}
\Phi^{-1}(X^\prime_t)=\frac{\exp(-at)}{2}\left(X_0+2\int_{0}^{t}\exp(as)dW^\prime_s\right).
\end{equation}

\subsection{A two dimensional example}\label{section_2d}
Let us consider the SDE
\begin{equation}\label{equation_singular}
\left(\begin{array}{c}
dX_t\\
dY_t\\
dZ_t
\end{array}\right)=\left(
\begin{array}{c}
\frac{\alpha X_t}{X_t^2+Y_t^2}\\
\frac{-\alpha Y_t}{X_t^2+Y_t^2}\\
1
\end{array} \right)dt +
\left(\begin{array}{cc}
\frac{X_t^2-Y_t^2}{\sqrt{X_t^2+Y_t^2}} & 0\\
0 & \frac{X_t^2-Y_t^2}{\sqrt{X_t^2+Y_t^2}}\\
0 & 0
\end{array}\right) \cdot \left(\begin{array}{c}
dW^1_t\\
dW^2_t
\end{array}\right),
\end{equation}
where $\alpha \in \mathbb{R}$, that has been discussed in \cite{DMU2} (see also \cite{Craddock2009}, where a similar equation has been studied using Lie's point symmetries of the related Kolmogorov equation).  Solving the determining equations for the quasi Doob symmetries of this SDE, we find
$$
V_1=\left(\left(\begin{array}{c}
\frac{y}{x^2+y^2}\\
\frac{x}{x^2+y^2}\\
0
\end{array} \right), \left(\begin{array}{cc}
0 & \frac{x^2-y^2}{(x^2+y^2)^2}\\
-\frac{x^2-y^2}{(x^2+y^2)^2} & 0
\end{array}\right),0, \left(\begin{array}{c}
0\\
0
\end{array} \right)\right)
$$

$$
V_2=\left(\left(\begin{array}{c}
x\\
y\\
0
\end{array} \right), \left(\begin{array}{cc}
0 & 0\\
0 & 0
\end{array}\right),0, \left(\begin{array}{c}
\frac {-2\alpha x}{(x^2-y^2)\sqrt{x^2+y^2}}\\
\frac {2\alpha y}{(x^2-y^2)\sqrt{x^2+y^2}}
\end{array} \right)\right)
$$

$$
V_3=\left(\left(\begin{array}{c}
0\\
0\\
1
\end{array} \right), \left(\begin{array}{cc}
0 & 0\\
0 & 0
\end{array}\right),0, \left(\begin{array}{c}
0\\
0
\end{array} \right)\right).
$$
Since  $[Y_1,Y_2]=2Y_1$, $[Y_1, Y_3]=0$ and $[Y_2, Y_3]=0$,  these symmetries form a solvable Lie algebra.
Therefore, as in the previous examples, we can compute a finite stochastic transformation $T=(\Phi, B, \eta, h)$ such that the transformed SDE $(\mu', \sigma')= E_T(\mu, \sigma)$ is in triangular form.
In particular for $\eta=1$,
\[
\Phi(x,y,z)=
\begin{pmatrix}
xy \\
\frac 12 \log (|x^2-y^2|)\\
z
\end{pmatrix},
\]

\[
B=\left(\begin{array}{cc}
\frac{y}{\sqrt{x^2+y^2}} & \frac{x}{\sqrt{x^2+y^2}}\\
\frac{-x}{\sqrt{x^2+y^2}} & \frac{y}{\sqrt{x^2+y^2}}\end{array}\right),
\]
 and
 \[
h=
\begin{pmatrix}
\frac {y -\frac {\alpha x}{x^2-y^2}}{\sqrt{x^2+y^2}} \\
\frac {y+\frac {\alpha x}{x^2-y^2}}{\sqrt{x^2+y^2}}
\end{pmatrix}
\]
 we can explicit compute the transformed SDE and we get
\[
\mu'=
\begin{pmatrix}
e^{2Y'_t} \\
-1\\
1
\end{pmatrix}
\]
 and
 \[
\sigma'=\left(\begin{array}{cc}
e^{2Y'_t} & 0\\
0& -1\\
0 & 0\end{array}\right).
\]
Therefore, the reduced SDE has the following triangular form
\begin{align}
dX^\prime_t&=\exp(2Y^\prime_t)dt+\exp(2Y^\prime_t)dW^\prime_{{1},{t}} \notag \\
dY^\prime_t&=-dt-dW^\prime_{{2},{t}} , \notag \\
dZ^\prime_t&=dt, \notag
\end{align}
and can be easily integrated.  Indeed
\begin{align}
Y^\prime_t&=Y^\prime_{{0}}-t-W^\prime_{{2},{t}} , \notag \\
X^\prime_t&=X^\prime_0+\int_{0}^{t}\exp(Y^\prime_0-s-W^\prime_{{2},{s}})ds+\int_{0}^{t}\exp(Y^\prime_0-s-W^\prime_{{2},{s}})dW^\prime_{{1},{s}} \notag \\
Z^\prime_t&=t, \notag
\end{align}
that is $(X^\prime_t,Y^\prime_t)$ is integrable according with Definition \ref{standard_reconstruction} with $O=X^\prime$ and $O^\prime=Y^\prime$.

By the reconstruction theorem (Theorem \ref{reconstruction}) and since there is no time change (i.e. $\eta=1$ ) we obtain
that, $\forall t \in [0,\mathcal{T}]$,
\begin{equation}
\mathbb{E}_{\mathbb{P}}[g(X_t,Y_t)]=\mathbb{E}_{\mathbb{P}}[g(\Phi^{-1}(X^\prime_t,Y^\prime_t))]=
\mathbb{E}_{\mathbb{Q}}\left[g(\Phi^{-1}(X^\prime_t,Y^\prime_t)) \frac{d\mathbb{P}}{d\mathbb{Q}}\Bigg|_{\mathcal{F}_{t}}	\right]
\end{equation}	
where
\[
\frac{d\mathbb{P}}{d\mathbb{Q}}\Bigg|_{\mathcal{F}_{\alpha_t}} =  \exp \left( -\int_0^{t} h^TB^{-1}(\Phi^{-1}(X^\prime_s,Y^\prime_s))d{W}^\prime_s\right)\times
\]
\[
\times \exp \left(+\frac{1}{2}\int_0^{t} h^Th(\Phi^{-1}(X^\prime_s,Y^\prime_s))ds \right).
\]
We calculate $\Phi^{-1}$ obtaining
\[
\Phi^{-1}(x^\prime,y^\prime,z^\prime)=
\begin{pmatrix}
\pm{\sqrt{\frac{\sqrt{\exp(4y^\prime)+4(x^\prime)^2}-\exp(2y^\prime)}{2}}} \\
\frac{x^\prime}{\pm\sqrt{\frac{\sqrt{\exp(4y^\prime)+4(x^\prime)^2}-\exp(2y^\prime)}{2}}} \\
z^\prime
\end{pmatrix},
\]
where we used that
\[
x^2+y^2=\sqrt{\exp(4y^\prime)+4(x^\prime)^2}.
\]
The Radon-Nikodym derivative up to time $t$ becomes
\[
\frac{d\mathbb{P}}{d\mathbb{Q}}\Bigg|_{\mathcal{F}_{\alpha_t}} =  \exp \left( -\int_0^{t}[ D_{1,s}d{W}^\prime_{1,s}+D_{2,s}d{W}^\prime_{2,s}]
-\frac{1}{2}\int_0^{t} 2D_{3,s} ds \right),
\]
with
\begin{align*}
D_{1,s}=&\frac{1}{\sqrt{\exp(4Y^\prime_s)+4(X^\prime_s)^2}}\left[\frac{\sqrt{\exp(4Y^\prime_s)
		+4(X^\prime_s)^2}}{2}+X^\prime_s-\frac{\exp(2Y^\prime_s)}{2}\right]\\
&+\frac{\alpha}{\exp(4Y^\prime_s)+4(X^\prime_s)^2}\left[\frac{\sqrt{\exp(4Y^\prime_s)
		+4(X^\prime_s)^2}}{2}-X^\prime_s+\frac{\exp(2Y^\prime_s)}{2}\right]
\end{align*}
\begin{align*}
D_{2,s}=&\frac{1}{\sqrt{\exp(4Y^\prime_s)+4(X^\prime_s)^2}}\left[\frac{\sqrt{\exp(4Y^\prime_s)
		+4(X^\prime_s)^2}}{2}-X^\prime_s-\frac{\exp(2Y^\prime_s)}{2}\right]\\
&+\frac{\alpha}{\exp(4Y^\prime_s)+4(X^\prime_s)^2}\left[\frac{\sqrt{\exp(4Y^\prime_s)
		+4(X^\prime_s)^2}}{2}+X^\prime_s+\frac{\exp(2Y^\prime_s)}{2}\right]
\end{align*}
\begin{align*}
D_{3,s}=&\frac{1}{\sqrt{\exp(4Y^\prime_s)+4(X^\prime_s)^2}}\left[\left(\frac{\sqrt{\exp(4Y^\prime_s)
		+4(X^\prime_s)^2}}{2}-\frac{\exp(2Y^\prime_s)}{2}\right)\right.\\
&\left.+\frac{\alpha^2}{\exp(4Y^\prime_s)+4(X^\prime_s)^2}\left(\frac{\sqrt{\exp(4Y^\prime_s)
		+4(X^\prime_s)^2}}{2}+\frac{\exp(2Y^\prime_s)}{2}\right)\right].
\end{align*}
Finally we get
\begin{multline*}
\mathbb{E}_{\mathbb{P}}[g(X_t,Y_t)]=\\
=\mathbb{E}_{\mathbb{Q}}\left[g\left(\pm{\sqrt{\frac{\sqrt{\exp(4Y^\prime_t)+4(X^\prime_t)^2}+\exp(2Y^\prime_t)}{2}}},{\pm\sqrt{\frac{\sqrt{\exp(4Y^\prime_t)+4(X^\prime_t)^2}-\exp(2Y^\prime_t)}{2}}}\right) \frac{d\mathbb{P}}{d\mathbb{Q}}\Bigg|_{\mathcal{F}_{t}}\right].
\end{multline*}

\section*{Conclusions and future developments}

In this paper we extend an application of (standard) Lie symmetry analysis of differential equations to the case of general symmetries of stochastic differential equations introduced in \cite{DMU3} and we prove some properties of these new general infinitesimal symmetries (namely that they form a Lie algebra). The main results of the paper are  reduction and reconstruction procedures for symmetric SDE in this general framework. Moreover, we introduce the notion of quasi Doob transformations and symmetries, useful for the mentioned reconstruction procedure. Future developments of this research will be certainly the applications of the previous theory to numerical integration of SDEs (as done in \cite{DeUg2017} for strong symmetries, see also \cite{AlDeMoUg2019}). In this direction the notion of quasi Doob symmetries seems to be very useful, since the reconstruction formula obtained above contains only Riemann integration in the exponential change of measure. From a theoretical point of view, another interesting  application of the general setting developed in this paper is the problem of finding  some new explicit formulas for generic symmetric processes similar to the ones discussed in Malliavin calculus for Brownian motion.

\section*{Acknowledgments}

We thank Elena Cattarin for the discussions on the topic and the help in a part of the present work. The first and third author are funded by the DFG under Germany’s Excellence Strategy - GZ 2047/1, project-id 390685813.

\bibliographystyle{plain}
\bibliography{doob_rec_def}

\end{document}